\documentclass{amsart}
\usepackage[utf8]{inputenc}
\usepackage[dvipsnames]{xcolor}
\usepackage{fullpage,hyperref,cleveref,mathtools,verbatim,tikz,amssymb,amsmath}
\hypersetup{colorlinks, citecolor=red, filecolor=black, linkcolor=blue, urlcolor=blue}
\usepackage{dsfont}
\usepackage[normalem]{ulem} 
\usetikzlibrary{calc}
\usepackage{diagbox}
\usepackage{pict2e}
\usepackage{tikz-cd}
\usepackage{float}
\usepackage{shuffle}
\usepackage[shortlabels]{enumitem}
\usepackage{soul}
\usepackage{subcaption}
\usepackage{xr} 

\setlist[itemize]{noitemsep, topsep=0pt}
\setlist[enumerate]{noitemsep, topsep=0pt}
\newcommand{\seqnum}[1]{\href{http://oeis.org/#1}{\underline{#1}}}

\definecolor{navy}{rgb}{0.0,0.0,0.6}

\newcommand{\Gabe}[1]{\textcolor{red}{[Gabe: #1]}}

\newcommand{\red}[1]{\textcolor{red}{#1}}

\newcommand{\N}{\mathbb{N}}
\newcommand{\Z}{\mathbb{Z}}
\renewcommand{\P}{\mathbb{P}}
\newcommand{\R}{\mathbb{R}}
\renewcommand{\S}{\mathfrak{S}}

\newcommand{\CC}{\mathbb{C}}
\newcommand{\eqpd}{\, .}
\newcommand{\eqcom}{\, ,}

\DeclareMathOperator{\des}{des}
\DeclareMathOperator{\Des}{Des}
\DeclareMathOperator{\asc}{asc}
\DeclareMathOperator{\Asc}{Asc}

\DeclareMathOperator{\inv}{inv}
\DeclareMathOperator{\Inv}{Inv}
\DeclareMathOperator{\maj}{maj}

\DeclareMathOperator{\dtop}{dtop}
\DeclareMathOperator{\itop}{itop}
\DeclareMathOperator{\sinv}{sinv}
\DeclareMathOperator{\sdes}{sdes}
\DeclareMathOperator{\bdes}{bdes}
\DeclareMathOperator{\binv}{binv}
\DeclareMathOperator{\tie}{tie}
\DeclareMathOperator{\area}{area}

\DeclareMathOperator{\PF}{PF}
\DeclareMathOperator{\UPF}{UPF} 

  \DeclareMathOperator{\Comp}{Comp}
    \DeclareMathOperator{\WComp}{WComp}
\DeclareMathOperator{\pinv}{pinv}
\DeclareMathOperator{\Fub}{Fub}
\DeclareMathOperator{\ch}{ch}
\DeclareMathOperator{\con}{con}
\DeclareMathOperator{\ps}{ps}

\newcommand{\defterm}{\textbf} 
\DeclareMathOperator{\E}{\mathbb{E}}
\DeclareMathOperator{\F}{\mathcal{F}}
\DeclareMathOperator{\C}{\mathfrak{C}}
\DeclareMathOperator{\Exp}{Exp}

\newcommand{\qbinom}[2]{\genfrac{[}{]}{0pt}{}{#1}{#2}}

\newcommand{\lp}{\left(}
\newcommand{\rp}{\right)}
\newcommand{\lb}{\left[}
\newcommand{\rb}{\right]}

\newtheorem{theorem}{Theorem}[section]
\newtheorem{corollary}[theorem]{Corollary}
\newtheorem{proposition}[theorem]{Proposition}
\newtheorem{problem}[theorem]{Problem}
\newtheorem{lemma}[theorem]{Lemma}

\theoremstyle{definition}
\newtheorem{definition}[theorem]{Definition}
\newtheorem{remark}[theorem]{Remark}
 \newtheorem{example}[theorem]{Example}

\title{Inversions in parking functions} 

\author[Celano]{Kyle Celano}\address[Celano]{Department of Mathematics, Wake Forest University, NC}
\email{\url{celanok@wfu.edu}}

\author[Elder]{Jennifer Elder}
\address[Elder]{Department of Computer Science, Mathematics and Physics, Missouri Western State University, St. Joseph, MO 64507}
\email{\textcolor{blue}{\href{mailto:jelder8@missouriwestern.edu}{jelder8@missouriwestern.edu}}} 

\author[Hadaway]{Kimberly P. Hadaway}
\address[Hadaway]{Department of Mathematics, Iowa State University, Ames, IA 50010}
\email{\textcolor{blue}{\href{mailto:kph3@iastate.edu}{kph3@iastate.edu}}}

\author[Harris]{Pamela E. Harris}
\address[Harris]{Department of Mathematical Sciences, University of Wisconsin-Milwaukee, Milwaukee, WI 53211}
\email{\textcolor{blue}{\href{mailto:peharris@uwm.edu}{peharris@uwm.edu}}}

\author[Priestley]{Amanda Priestley}
\address[Priestley]{Department of Computer Science, The University of Texas at Austin, Austin, TX 78712}
\email{\textcolor{blue}{\href{mailto:amandapriestley@utexas.edu}{amandapriestley@utexas.edu}}} 

\author[Udell]{Gabe Udell}
\address[Udell]{Department of Mathematics, Cornell University, Ithaca, NY 14850}
\email{\textcolor{blue}{\href{mailto:gru5@cornell.edu}{gru5@cornell.edu}}}

\begin{document}

\begin{abstract}
    In this paper, we obtain a $q$-exponential generating function for inversions on parking functions via symmetric function theory and also through a direct bijection to rooted labeled forests. We then apply these techniques to unit interval parking functions to give analogous results. We conclude by introducing a  probabilistic approach through which we obtain formulas for the total number of inversions and several other statistics across all parking functions and other sets of words closed under rearrangement. 
\end{abstract}

\maketitle

\section{Introduction}
Throughout, let $\N=\{1,2,\dots\}$, and if $n\in\N$, then $[n]=\{1,2,\ldots,n\}$. 
Also, let $\S_n$ denote the set of permutations of $[n]$.
Konheim and Weiss \cite{Konheim1966} defined parking functions as follows: consider a one-way street with $n$ parking spots (labeled in increasing order) and $n$ cars, each of which has a preferred spot. 
As each car parks, it drives to its preferred spot and parks there if it is unoccupied. 
If that spot is occupied, then the car continues driving and parks in the next available spot, if any exists.  
We encode the information of the preferred spots as a positive integer-valued \defterm{preference list} $\alpha = (\alpha_1,\alpha_2,\dots,\alpha_n)\in [n]^n$, where for each $i\in[n]$, the positive integer $\alpha_i$ indicates the preferred spot of car $i$. 
We say the preference list $\alpha$ is a \defterm{parking function} if all cars are able to park using the aforementioned parking rule. 
The \defterm{outcome permutation} of $\alpha$, denoted $\pi(\alpha)$, is defined by setting $\pi(\alpha)(i)=j$ if car $j$ parks in spot $i$ for each $i\in [n]$. 
We let $\PF_n$ be the set of all parking functions of length $n$. 
Konheim and Weiss establish
that the number of parking functions of length $n$ is  
$(n+1)^{n-1}$, see \cite[Lemma 1]{Konheim1966}. 
A second proof of this count was given by 
Pollak by parking the cars on a circle with one new parking spot, see \cite{foata1974mappings}.

Introduced by Hadaway \cite{bib:HadawayUndergradThesis},  a \defterm{unit interval parking function} of length $n$ is a parking function $\alpha\in \PF_n$ such that $\pi(\alpha)^{-1}(i)-\alpha_i\leq 1$ for all $i\in[n]$. 
In other words, these are parking functions in which, for all $i\in[n]$, car $i$ parks either in its preference $a_i$ or in the spot after its preference $a_i+1$.
Let $\UPF_n$ denote the set of unit interval parking functions of length $n$. 
Hadaway \cite{bib:HadawayUndergradThesis} proves that $|\UPF_n|=\Fub_n$, which is the \textbf{$n^{th}$ Fubini number} and counts the number of ordered set partitions of $[n]$  \cite[\seqnum{A000670}]{OEIS}. 
Unit interval parking functions have been studied in connection to enumerating Boolean intervals of the weak Bruhat order of the symmetric group and in connection to the faces of the permutohedron \cite{unit_perm,elder2024parking}.

We are interested in the study of inversions of parking functions and unit interval parking functions. 
Recall that for a word $w\in \N^n$, an \defterm{inversion} is a pair $(i,j)$ of integers in $[n]$ such that $i<j$ and $w_i>w_j$.
We denote the set of inversions of a word $w$ by $\Inv(w)$ and let $\inv(w)=|\Inv(w)|$. 
Treating parking functions as words, for each $n\in\N$, we define the polynomials
\[\PF_n(q)=\sum_{\alpha\in \PF_n}q^{\inv(\alpha)}\quad \mbox{and}\quad\UPF_n(q)=\sum_{\alpha\in \UPF_n}q^{\inv(\alpha)}, \]
and set $\PF_0(q)=\UPF_0(q)=1$. 
The inversion generating function $\PF_n(q)$ appeared incidentally in \cite[Section 8]{NadTew23} in their study of remixed Eulerian numbers and as we explain in \Cref{sec:inversions-in-pfs}, it appears as a specialization of certain symmetric functions studied in \cite{haiman_conjectures_1994}. Beyond this, little  has been said about inversions in parking functions directly. 

Define
\[\Exp_q(z)=\sum_{n\geq 0}q^{\binom{n}{2}}\frac{z^n}{[n]_q!}\quad \mbox{and }\quad\exp_q(z)=\sum_{n\geq 0}\frac{z^n}{[n]_q!},\]
where $[n]_q=1+q+\cdots+q^{n-1}$ and $[n]_q!=[n]_q[n-1]_q\cdots[1]_q$.
With these definitions, using combinatorial techniques, we derive the following generating functions. 
\begin{theorem}\label{eq:pf-upf-inv-gen-functs} The $q$-exponential generating functions for $\PF_n(q)$ and $\UPF_n(q)$ are 
    \[\sum_{n\geq 0}\PF_n(q)\frac{z^{n+1}}{[n]_{q}!}=(z\Exp_{q}(-z))^{\langle -1\rangle }\quad \mbox{and }\quad\sum_{n\geq 0}\UPF_n(q)\frac{z^{n}}{[n]_{q}!}=\frac{1}{2-\exp_q(z)},\]
    where $F(z)^{\langle -1\rangle}$ denotes the compositional inverse of $F(z)$.
\end{theorem}
 \Cref{eq:pf-upf-inv-gen-functs} gives $q$-analogues of  classical results of Konheim and Weiss \cite{Konheim1966} and Cayley \cite{cayley1856}, respectively: 
\begin{equation}\label{eq:results2}
    \sum_{n\geq 0}|\PF_n|\frac{z^{n+1}}{n!}=(ze^{-z})^{\langle -1\rangle }\quad \mbox{and } \quad \sum_{n\geq 0}|\UPF_n|\frac{z^{n}}{n!}=\frac{1}{2-e^z},
\end{equation}
where the latter result follows from the bijection between unit interval parking functions and Cayley permutations (see \cite{bib:HadawayUndergradThesis}).
We also show that the equations in \eqref{eq:results2} follow from a more general theory of symmetric functions. The main result is as follows.
\begin{theorem}\label{eq:pf-upf-sym-gen-functs}There are $\S_n$-modules $\CC[\PF_n]$ and $\CC[\UPF_n]$ such that 
\begin{equation}\label{eq:frobenius_char}
    \sum_{n\geq 0}\ch\CC[\PF_n]z^{n+1}=(zE(-z))^{\langle -1\rangle }\quad \mbox{and } \quad \sum_{n\geq 0}\ch \CC[\UPF_n]z^n=\frac{1}{2-H(z)} \ ,
\end{equation}
where $\ch$ denotes the Frobenius characteristic map from $\S_n$-modules to symmetric functions of degree $n$.
Here,
$E(z)=\sum_{n\geq 0}e_nz^n$ and $H(z)=\sum_{n\geq 0}h_nz^n$
are generating functions for the elementary symmetric functions and the complete homogeneous symmetric functions, respectively.\end{theorem} 
The Frobenius image in \eqref{eq:frobenius_char} for parking functions was computed algebraically by Haiman \cite{haiman_conjectures_1994}. In \Cref{sec:inversions-in-pfs}, we provide a new combinatorial proof using labeled rooted forests, which we construct in \Cref{sec:labeled rooted forests}. Then, using a modified version of stable principal specialization, \Cref{eq:pf-upf-sym-gen-functs} gives a second proof of \Cref{eq:pf-upf-inv-gen-functs}. For unit interval parking functions, we provide two ways of computing its Frobenius characteristic in \Cref{sec:inversions-in-upfs}: one by specializing a formula of Stanley and the other by more elementary arguments.

\begin{table}[h]
\[\arraycolsep=1.4pt\def\arraystretch{1.5}
\begin{array}{|c|c||c|c|c|}\hline
    &\S_n&\PF_n& \mbox{ Theorem } &\mbox{ OEIS/Sequence }\\\hline
        \inv&\frac{n!n(n-1)}{4}&\frac{n(n+1)^{n-2}}{2}\binom{n}{2}& \mbox{ \Cref{thm:total-number-of-inversions-in-parking-functions}} & \seqnum{A386011}^{\text{\textdagger}}
        \\\hline
    \des&\frac{n!(n-1)}{2}&\binom{n}{2}(n+1)^{n-2}& \text{\cite[Theorem 10]{bib:DescentsInPF}} &
    \seqnum{A053507}
    \\\hline
    \des_1&\frac{n!}{2}&\frac{n}{2}(n+1)^{n-2}& 
    \text{\cite[Theorem 10]{bib:DescentsInPF}\red{*}}
    & \seqnum{A386015}^{\text{\textdagger}}
    \\\hline
    \tie&0&(n-1)(n+1)^{n-2}& \mbox{ \Cref{cor:ties in parking functions}} &\seqnum{A071720}\\\hline
    \tie_1&0&(n+1)^{n-2}& \mbox{\Cref{prop:parktie}} &  \seqnum{A007830}\\\hline
    \sdes&(n-1)(n-1)!&(n-1)(n+1)^{n-2}&  \mbox{ \Cref{prop:sdes_sinv_bdes_binvPF}}  & \seqnum{A071720}\\\hline
    \sdes_1&(n-1)!&(n+1)^{n-2}& \mbox{ Equation} \eqref{eq:exp-of-1-sdes-pf}\red{*}  & \seqnum{A007830}\\\hline
    \sinv&\frac{n!(n-1)}{2}&\binom{n}{2}(n+1)^{n-2}&  \mbox{ \Cref{prop:sdes_sinv_bdes_binvPF}} & \seqnum{A053507}  \\\hline
    \bdes&\binom{n-1}{2}(n-1)!&\binom{n-1}{2}(n+1)^{n-2}&  \mbox{ \Cref{prop:sdes_sinv_bdes_binvPF}} & \seqnum{A386860}^{\text{\textdagger}}\\\hline
    \bdes_1&\frac{n-2}{2}(n-1)!&\frac{n-2}{2}(n+1)^{n-2}&  \mbox{ \Cref{prop:sdes_sinv_bdes_binvPF}}\red{*}& \seqnum{A387047}^{\text{\textdagger}} \\\hline
    \binv&\binom{n-1}{2}\frac{n!}{2}& \frac{n}{4}(n-1)(n-2)(n+1)^{n-2}&  \mbox{\Cref{prop:sdes_sinv_bdes_binvPF} } & \seqnum{A386861}^{\text{\textdagger}}\\\hline
\end{array}\]
   \caption{Totals of certain statistics over $\S_n$ and $\PF_n$. The references to theorems and OEIS entries given refer to the results for parking functions in particular. The analogous results for the symmetric group are classical (e.g.\cite{Stern_1838, Terquem1838}), except those involving sdes and those listed in the rows below it, which are given by \Cref{prop:sdes_sinv_bdes_binvSN}. The references marked with (\red{*}), are the result of applying \Cref{thm:k-transitive-sn-invariant-inv-and-sum} to the statistics indicated. OEIS entries that were added as a result of this project are marked with (\text{\textdagger}). 
   For any OEIS entry which is unmarked, except those for descents, our project contributes a new combinatorial perspective in terms of parking functions.}
    \label{tab:stats-on-sn-and-pf} 
\end{table}

The study of the inversion statistic on $\PF_n$ and $\S_n$ has an important commonality: both set of words are $\S_n$-invariant. 
Hence, we focus studying the average behavior of  statistics in general on $\S_n$-invariant sets of words. 
With this approach in mind, in \Cref{sec:expectations}, we show that for any  $\S_n$-invariant subset $W\subset \N^n$ of positive integers, we have a curious equality: 
\begin{equation}\label{eq:sum-inv-ws-intro}\sum_{w\in W}\inv(w)=\frac{n}{2}\sum_{w\in W}\des(w),\end{equation}
where $\des(w)$ is the number of \defterm{descents} of $w$ i.e., the number of positions $i\in [n-1]$ such that $w_i>w_{i+1}$.

The intuition behind \eqref{eq:sum-inv-ws-intro} involves noting that the total number of inversions can be written as
the sum over all $i<j$ of the number of words with $(i,j)$ as an inversion.
The key\footnote{We credit Eyob Tsegaye for inspiring this approach.} is to use the $\S_n$-action to move the positions $(i,j)$ to the positions $(1,2)$. Then, the probability that $(i,j)$ is an inversion is the same as the probability that $(1,2)$ is an inversion. 
Using a probabilistic perspective, one can utilize this observation to obtain \eqref{eq:sum-inv-ws-intro}. 
The idea of taking a statistic that is computed by summing over locations, computing the statistic by considering each location separately, and finally using $\S_n$-invariance of $W$ to move each location to the beginning, is a technique that applies more broadly.  We introduce the notion of a \defterm{$k$-transitive function} to describe the class of statistics that this argument applies to, and in \Cref{thm:k-transitive-sn-invariant-inv-and-sum} we give a vast generalization of \eqref{eq:sum-inv-ws-intro}. 
In \Cref{sec:applications-of-linearity-of-expectation}, we explore applications of these ideas, and among our results we establish the results in \Cref{tab:stats-on-sn-and-pf}. 
For the interested reader, Campion Loth, Levet, Liu, Stucky, Sundaram, and Yin \cite{LLLSSY23} consider a similar problem for statistics on permutations in particular conjugacy classes. 

\section{Labeled rooted forests}\label{sec:labeled rooted forests}
The main objective of this section is to describe a statistic on rooted forests, which by way of a bijection to parking functions, maps to inversions of parking functions. 
We use standard definitions from graph theory for trees, forests, rooted trees, and rooted forests, for more on those topics we recommend \cite{DiestelGraph}.
To begin, recall that parking functions are in bijection with combinatorial objects called \defterm{labeled rooted forests} \cite{Konheim1966}, where a rooted forest is made up of rooted trees in which every vertex is given a unique integer. 
We denote the set of labeled rooted forests on vertex set $[n]$ by $\F_n$. 
If $T$ is a rooted labeled tree and $v$ is a nonroot vertex of $T$, define the \defterm{parent} of $v$, denoted $p(v)$, to be first element on the path from $v$ to the root of $T$. 
We say that $v$ is a \defterm{child} of $u$ if $p(v)=u$.
For a labeled rooted tree $T$, let $r(T)$ denote its root. 
The \defterm{subtrees} of a labeled rooted tree $T$ are labeled rooted subtrees $T_1,T_2,\dots,T_k$ such that $r(T_i)$ is adjacent in $T$ to $r(T)$ for each $i\in [k]$, where we order the trees so that $r(T_1)<r(T_2)<\cdots<r(T_k)$. 
An unlabeled rooted tree is called a \textbf{rooted plane tree} provided the children of each internal vertex are given a total ordering. 
A forest of rooted planar trees is a \textbf{rooted plane forest} 
if the forests are given a total ordering.
Next, we introduce technical definitions used in our proofs.
\begin{definition}
Let $T$ be a labeled rooted tree with root $r$. 
The \defterm{preorder traversal permutation} $w(T)$ of $T$ is defined recursively by setting 
\[w(T)=\begin{cases}
    r(T)&\text{if $T$ is the single vertex $r(T)$}\\
    r(T)\cdot w(T_1)\cdot w(T_2)\cdots w(T_k)&\text{if $T$ has subtrees $T_1,T_2,\dots,T_k$},
\end{cases}\]
where $u\cdot v$ denotes concatenation of words $u$ and $v$. 

 For $F\in \F_n$, define the \defterm{preorder traversal permutation} of $F$ to be the permutation $w(F)$ on $\{0\}\cup [n]$ 
 defined as follows: Suppose $F$ has trees $T_1,T_2,\dots,T_k$ with roots $r_1<r_2<\cdots<r_k$. 
 Add a new vertex 0 to $F$ and connect the roots of $T_1,T_2,\dots,T_k$ to $0$ to make a rooted tree $T(F)$
 rooted at 0.  Then define $w_F=w(T(F))$.
\end{definition}

Fran\c{c}on \cite{francon_acyclic_1975} defines a bijection $\rho:\F_n\to \PF_n$ by
    \begin{align}
    \rho(F)=(w_F^{-1}(p(1)),w_F^{-1}(p(2)),\dots, w_F^{-1}(p(n))).\label{def:rho}
    \end{align}
 \Cref{fig:root-forest-bij-pf} provides an example of $\rho$, 
  where, for instance, the entries in positions $3,10,$ and $14$ of $\rho(F)$ are each $2$ because the nodes labeled $3,10,$ and $14$ in $F$ are the children of the node labeled $5$, which is in position $2$ of $w_F$. 
  
\begin{figure}
    \centering
    \begin{subfigure}[b]{0.4\textwidth}
    \centering
        \begin{tikzpicture}[yscale=-1, scale=.8]
        \node (0) at (0,0) {\textcolor{white}{0}};
        \node (1) at (2,3) {1};
        \node (2) at (-3.5,3) {2};
        \node (3) at (-3,2) {3};
        \node (4) at (-1.5,3) {4};
        \node (5) at (-1.5,1) {5};
        \node (6) at (-.5,3) {6};
        \node (7) at (0,1) {7};
        \node (8) at (1.5,1) {8};
        \node (9) at (-2.5,3) {9};
        \node (10) at (-2,2) {10};
        \node (11) at (0,2) {11};
        \node (12) at (1,2) {12};
        \node (13) at (2,2) {13};
        \node (14) at (-1,2) {14};
        \draw (5)--(3)--(2);
        \draw (3)--(9);
        \draw (5)--(10);
        \draw (5)--(14)--(4);
        \draw (14)--(6);
        \draw (7)--(11);
        \draw (8)--(13)--(1);
        \draw (8)--(12);
    \end{tikzpicture}
    \caption{Labeled rooted forest $F$}
    \label{fig:rooted-forest}
    \end{subfigure}
     \begin{subfigure}[b]{0.4\textwidth}
    \centering
        \begin{tikzpicture}[yscale=-1, scale=.8]
        \node (0) at (0,0) {1};
        \node (1) at (2,3) {15};
        \node (2) at (-3.5,3) {4};
        \node (3) at (-3,2) {3};
        \node (4) at (-1.5,3) {8};
        \node (5) at (-1.5,1) {2};
        \node (6) at (-.5,3) {9};
        \node (7) at (0,1) {10};
        \node (8) at (1.5,1) {12};
        \node (9) at (-2.5,3) {5};
        \node (10) at (-2,2) {6};
        \node (11) at (0,2) {11};
        \node (12) at (1,2) {13};
        \node (13) at (2,2) {14};
        \node (14) at (-1,2) {7};
        \draw (0)--(5)--(3)--(2);
        \draw (3)--(9);
        \draw (5)--(10);
        \draw (5)--(14)--(4);
        \draw (14)--(6);
        \draw (0)--(7)--(11);
        \draw (0)--(8)--(13)--(1);
        \draw (8)--(12);
    \end{tikzpicture}
    \caption{Preorder traversal of $T(F)$}
    \label{fig:rooted-forest-preorde-traversal}
    \end{subfigure}
  \caption[This is the caption; This is the second line]
    {\tabular[t]{@{}l@{}}$w(F)=(0,5,3,2,9,10,14,4,6,7,11,8,12,13,1)$\\ $\rho(F)=(14,3,2,7,1,7,1,1,3,2,10,12,12,2)$ \endtabular}\label{fig:root-forest-bij-pf}
\end{figure}
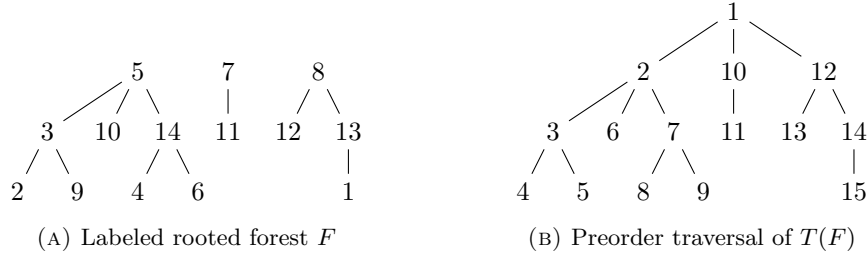

    \begin{definition}
         Suppose $F\in \F_n$. The pair $(i,j)$ of integers $i,j\in [n]$ is called a \defterm{parental preorder inversion} of $F$ provided $i<j$ and $w_F^{-1}(p(i))>w_F^{-1}(p(j))$. We denote the number of parental preorder inversions of $F$ by $\pinv(F)$.
    \end{definition}

    For a vertex, $i$, of a rooted labeled forest, $F$, let $p(i)$ denote the parent of $i$. Observe that $(i,j)$ is an inversion of $\rho(F)$ if and only if $i<j$ and $w_F^{-1}(p(i))>w_F^{-1}(p(j))$. 
    Hence, $\pinv(F)=\inv(\rho(F))$ and thus we have shown:
   \begin{proposition}\label{prop:parental-preorder-inversion-gf}
   For all $n\geq 1$, 
    $    \PF_n(q)=\displaystyle\sum_{F\in \F_n}q^{\pinv(F)}$.
   \end{proposition}

The constant term of $\PF_n(q)$ is the Catalan number $C_n=\frac{1}{n+1}\binom{2n}{n}$ \cite[\seqnum{A000108}]{OEIS} because on one hand the constant term is counting the weakly increasing plane forests and on the other hand, the constant term is counting unlabeled plane forests (see \cite{stanley2015catalan}).
 
\begin{figure}
    \centering
    \begin{subfigure}[b]{0.4\textwidth}
    \centering
        \begin{tikzpicture}[yscale=-1,every node/.style={shape=circle,draw,fill=black,inner sep=1pt,outer sep=5pt}]
        \node (1) at (2,3) {};
        \node (2) at (-3.5,3) {};
        \node (3) at (-3,2) {};
        \node (4) at (-1.5,3) {};
        \node (5) at (-1.5,1) {};
        \node (6) at (-.5,3) {};
        \node (7) at (0,1) {};
        \node (8) at (1.5,1) {};
        \node (9) at (-2.5,3) {};
        \node (10) at (-2,2) {};
        \node (11) at (0,2) {};
        \node (12) at (1,2) {};
        \node (13) at (2,2) {};
        \node (14) at (-1,2) {};
        \draw (5)--(3)--(2);
        \draw (3)--(9);
        \draw (5)--(10);
        \draw (5)--(14)--(4);
        \draw (14)--(6);
        \draw (7)--(11);
        \draw (8)--(13)--(1);
        \draw (8)--(12);
    \end{tikzpicture}
    \caption{Unlabeled rooted plane forest $U$}
    \label{fig:unlabled-plane-rooted-forest}
    \end{subfigure}
    \begin{subfigure}[b]{0.4\textwidth}
    \centering
        \begin{tikzpicture}[yscale=-1]
        \node (1) at (2,3) {14};
        \node (2) at (-3.5,3) {7};
        \node (3) at (-3,2) {4};
        \node (4) at (-1.5,3) {9};
        \node (5) at (-1.5,1) {1};
        \node (6) at (-.5,3) {10};
        \node (7) at (0,1) {2};
        \node (8) at (1.5,1) {3};
        \node (9) at (-2.5,3) {8};
        \node (10) at (-2,2) {5};
        \node (11) at (0,2) {11};
        \node (12) at (1,2) {12};
        \node (13) at (2,2) {13};
        \node (14) at (-1,2) {6};
        \draw (5)--(3)--(2);
        \draw (3)--(9);
        \draw (5)--(10);
        \draw (5)--(14)--(4);
        \draw (14)--(6);
        \draw (7)--(11);
        \draw (8)--(13)--(1);
        \draw (8)--(12);
    \end{tikzpicture}
    \caption{$F\in \mathcal{F}_{14}$ with $\pinv(F)=0$}
    \label{fig:rooted-forest-no-pinv}
    \end{subfigure}
    \caption{}
    \label{fig:labeling-UPRF}
   
\end{figure}

\section{Inversions in parking functions}\label{sec:inversions-in-pfs}
In this section, we  investigate inversions for parking functions by combining the bijection $\rho:\F_n\to \PF_n$ with the classical parking function symmetric function of Haiman \cite{haiman_conjectures_1994}. 
We assume a basic understanding of symmetric functions and $\S_n$-representation theory e.g. \cite[Chapter 7]{stanley1999enumerative} and \cite{sagan2013symmetric}. 

We set the following notation for this and subsequent sections. A \textbf{weak composition} of $n$ with $k$ parts is a sequence $(c_1,c_2,\dots,c_k)$ of nonnegative integers satisfying $\sum_{i=1}^k c_i=n$. 
A \textbf{composition} is a weak composition with only positive parts. 
We let $\ell(c)$ be the number of parts of a weak composition $c$.
Let $\Comp(n,k)$ and $\WComp(n,k)$ to be the set of compositions and weak compositions, respectively, of $n$ with $k$ parts. 
Also, let $\Comp(n)=\bigcup_{k}\Comp(n,k)$ and likewise $\WComp(n)=\bigcup_{k}\WComp(n,k)$. 
Given a word $w\in \N^n$ with largest value $m$, the \textbf{content} of $w$, denoted by $\con(w)=(c_1,c_2,\ldots,c_m)\in\WComp(n)$ is defined by setting $c_i$ to be the number of times the value $i$ appears in $w$ for each $i\in[m]$.

Recall that the \defterm{complete homogeneous symmetric function} 
$h_c$ is defined by $h_c=h_{c_1}h_{c_2}\cdots$ where for all positive integers $k$, we have 
\[h_k\coloneqq \sum_{1\leq i_1\leq i_2\leq \cdots\leq i_k}x_{i_1}x_{i_2}\cdots x_{i_k},\]
and $h_0\coloneqq 1$. 
Similarly, the \defterm{elementary symmetric function} $e_c$ is defined by $e_c=e_{c_1}e_{c_2}\cdots$, where for all positive integers $k$, we have 
\[e_k\coloneqq \sum_{1\leq i_1< i_2<\cdots<i_k}x_{i_1}x_{i_2}\cdots x_{i_k},\]
and $e_0\coloneqq1$. 
Finally, the \defterm{power-sum symmetric functions} 
is defined by $p_c=p_{c_1}p_{c_2}\cdots $, where for all positive integers $k$, we have
\[p_k\coloneqq \sum_{i\geq 1}x_i^k,\]
and $p_0\coloneqq 1$.

If $X$ is a finite set, let $\CC[X]$ denote the vector space over $\CC$ with basis $X$. 
If a group $G$ acts on $X$, we abuse notation and let $\CC[X]$ also denote the standard $G$-module defined by the action of $G$ on the basis $X$ (thus extending by linearity to the entire vector space). 
For an $\S_n$-module $V$ with character $\chi$, the \defterm{Frobenius image} is 
\[\ch(V)=\frac{1}{n!}\sum_{\sigma\in \S_n}\chi(\sigma)p_{cyc(\sigma)},\]
where $cyc(\sigma)$ denotes the cycle type of $\sigma$. The map $\ch$ from the set of $\S_n$-modules to symmetric functions is called the \defterm{Frobenius characteristic map}. 

\subsection{On the \texorpdfstring{$\S_n$}{}-module of parking functions and rooted labeled forests}
The symmetric group $\S_n$ acts on $\PF_n$ via 
\[\sigma\cdot \alpha  =(\alpha_{\sigma^{-1}(1)},\alpha_{\sigma^{-1}(2)},\dots,\alpha_{\sigma^{-1}(n)}) \quad \text{for all $\sigma\in \S_n$, $\alpha\in \PF_n$,} \]
making the module $\CC[\PF_n]$. The set $\F_n$ can also be made into an $\S_n$-module $\CC[\F_n]$ as follows. 
For $F\in \F_n$ and an index $i\in [n-1]$, define $s_i(F)$ to be the rooted labeled forest obtained from $F$ by applying the simple transposition $(i,i+1)$ to the labels and swapping the labels of $i$ and $i+1$. Then define the $\S_n$-action by setting
\begin{equation}\label{def: sn action on forests}
    (i,i+1)\cdot F=\begin{cases}
       s_i(F)&\mbox{if }p(i)\neq p(i+1)\\
       F&\mbox{if }p(i)=p(i+1)
    \end{cases}
\end{equation}
for each $i\in [n-1]$ and $F\in \F_n$. 
We can show for all $i\in [n-1]$ that
\begin{enumerate}
     \item $ (i,i+1)\cdot ( (i,i+1)\cdot F)=F$,
     \item $ (i,i+1)\cdot ( (i+1,i+2)\cdot ( (i,i+1)\cdot F))= (i+1,i+2)\cdot ( (i,i+1)\cdot ( (i+1,i+2)\cdot F))$, and
     \item $(i,i+1)\cdot ((j,j+1)\cdot F)=(j,j+1)\cdot ((i,i+1)\cdot F)$ if $|j-i|>1$.
 \end{enumerate}
We establish the above items in the following  technical result.
\begin{lemma}
    \Cref{def: sn action on forests} defines an $\S_n$-action on $\F_n$.
\end{lemma}
\begin{proof}
    This follows from checking that all relations of the generators $(i,i+1)$ of $\S_n$ are satisfied by the action; specifically:
    \begin{enumerate}
    \item $ (i,i+1)\cdot ( (i,i+1)\cdot F)=F$,
    \item $ (i,i+1)\cdot ( (i+1,i+2)\cdot ( (i,i+1)\cdot F))= (i+1,i+2)\cdot ( (i,i+1)\cdot ( (i+1,i+2)\cdot F))$, and
    \item $(i,i+1)\cdot ((j,j+1)\cdot F)=(j,j+1)\cdot ((i,i+1)\cdot F)$ if $|j-i|>1$.
\end{enumerate}
We verify each:
\begin{enumerate}
    \item Since $s_i(s_i(F))=F$, this is true regardless of the relation of $p(i)$ and $p(i+1)$.
    \item 
\begin{enumerate}
    \item $p(i)=p(i+1)=p(i+2)$: This is clear since everything is fixed.
    \item $p(i)=p(i+1)\neq p(i+2)$: 
    \begin{align*}  
   (i,i+1)\cdot ( (i+1,i+2)\cdot ( (i,i+1)\cdot F))&=(i,i+1)\cdot ( (i+1,i+2) \cdot F)\\
   &=(i,i+1)\cdot ( s_{i+1}(F))\\
   &=s_is_{i+1}(F),\\
   \intertext{and}
    (i+1,i+2)\cdot ( (i,i+1)\cdot ( (i+1,i+2)\cdot F)&=(i+1,i+2)\cdot ( (i,i+1)\cdot s_{i+1}(F))\\
    &=(i+1,i+2)\cdot s_i s_{i+1}(F))\\
    &=s_is_{i+1}(F).\end{align*}
    \item  $p(i)\neq p(i+1)= p(i+2)$: This is similar to the previous case.
    \item $p(i)\neq p(i+1)\neq p(i+2)$:
    \begin{align*}(i,i+1)\cdot ( (i+1,i+2)\cdot ( (i,i+1)\cdot F))    &=(i,i+1)\cdot ( (i+1,i+2) \cdot s_i(F))\\
    &=(i,i+1)\cdot s_{i+1} s_i(F))\\
    &=s_i s_{i+1} s_i(F)),\\
    \intertext{and}
        (i+1,i+2)\cdot ( (i,i+1)\cdot ( (i+1,i+2)\cdot F))&=(i+1,i+2)\cdot ( (i,i+1)\cdot s_{i+1}(F)\\
        &=(i+1,i+2)\cdot (  s_is_{i+1}(F)\\
        &=s_{i+1} s_is_{i+1}(F).
    \end{align*}
    Now, $s_{i}s_{i+1}s_i(F)=s_{i+1}s_is_{i+1}(F)$ because \[(i,i+1)(i+1,i+2)(i,i+1)=(i+1,i+2)(i,i+1)(i+1,i+2).\]
\end{enumerate}
\item This is straightforward to compute in all 4 cases. \qedhere
\end{enumerate}
\end{proof}

Hence, this defines an $\S_n$-action on $\F_n$.  The map $\rho$, defined in \Cref{def:rho}, then extends to a map of vector spaces $\rho:\CC[\F_n]\to \CC[\PF_n]$ via linearity.

    Recall that the \defterm{area} of  $\alpha=(\alpha_1,\alpha_2,\ldots,\alpha_n)\in \PF_n$ is the quantity
    \[\area(\alpha)=\binom{n+1}{2}-\sum_{i=1}^n\alpha_i.\]
    Under the standard bijection of parking functions to labeled Dyck paths (see \cite{YanPFs}, page 54), the quantity $\area(\alpha)$ can be interpreted as the number of full squares below the Dyck path and above the main diagonal; see \Cref{fig:area_example_dyck_path}.
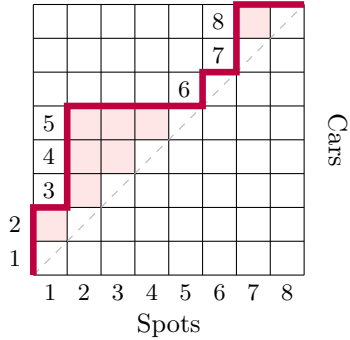
\begin{figure}[htb]\centering
\begin{tikzpicture}[scale=0.45]
\fill[color=red!10!white]  (0, 1) rectangle (1, 2);
\fill[color=red!10!white]  (1, 2) rectangle (2, 5);
\fill[color=red!10!white]  (2, 3) rectangle (3, 5);
\fill[color=red!10!white]  (3, 4) rectangle (4, 5);
\fill[color=red!10!white]  (6, 7) rectangle (7, 8);

\foreach \x in {0,...,8}
    \draw (\x, 0)--(\x,8) (0,\x)--(8,\x);

\draw[dashed,color=white!70!black] (0, 0)--(8, 8);

\node[] (x-axis-label) at (4, -1.5) {Spots};
\node[rotate=-90] (y-axis-label) at (9, 4) {Cars};

\foreach \x in {1,...,8} {
    \node at ({\x - .5}, -.5) {\small$\x$};
    }

    \draw[line width=2.5pt,color=purple] (0, 0)
    -- node[midway,left,color=black] {\small $1$} (0,1)
    --node[midway,left,color=black] {\small $2$} (0, 2)--(1, 2)
    --node[midway,left,color=black] {\small $3$} (1, 3)
    --node[midway,left,color=black] {\small $4$} (1, 4)
    -- node[midway,left,color=black] {\small $5$} (1, 5)--(5, 5)
    --node[midway,left,color=black] {\small $6$} (5, 6)--(6, 6)
    --node[midway,left,color=black] {\small $7$} (6, 7)
    --node[midway,left,color=black] {\small $8$} (6, 8)
    --(8, 8);
    \end{tikzpicture}

\caption{The parking function $(1,1,2,2,2,6,7,7)$ as a labeled Dyck path with the area shaded.}
\label{fig:area_example_dyck_path}
\end{figure}

    Then, if $\PF^{(i)}_{n}$ is the set of parking functions with area $i$, $\CC[\PF_n]$ has the structure of a graded module
    \[\CC[\PF_n]=\bigoplus_{i\geq 0}\CC[\PF^{(i)}_{n}] \eqpd \]
    Define the \defterm{area} of a rooted labeled forest $F\in \F_n$ as the quantity
    \[\area(F)=\binom{n+1}{2}-\sum_{i=1}^n w^{-1}_F(p(i)),\]
 and define the set
   \[\F_n^{(i)}=\{F\in \F_n\mid \area(F)=i\}.\]
With these definitions at hand, we begin by establishing that $\CC[\F_n]$ is a graded module.

\begin{proposition}\label{isomorphism-Sn-module}
    The $\S_n$-module $\CC[\F_n]$ is a graded module
    \[\CC[\F_n]=\bigoplus_{i\geq 0}\CC[\F_{n}^{(i)}].
    \]
    Moreover, the map $\rho:\CC[\F_n]\to \CC[\PF_n]$ is an isomorphism of graded $\S_n$-modules.
\end{proposition}

\begin{proof}
    Suppose $\rho(F)=\alpha\in \PF_n$ for some $F\in \F_n$, where $\rho$ is as in \Cref{def:rho}. Let $i\in [n-1]$. If $p(i)= p(i+1)$, then $w_F^{-1}(p(i))=w_F^{-1}(p(i+1))$, so $(i,i+1)\cdot \rho(F)=\rho(F)$. 
    If $p(i)\neq p(i+1)$, then $w_F^{-1}(p(i))\neq w_F^{-1}(p(i+1))$, so $(i,i+1)\cdot \rho(F)=\rho(s_i(F))=\rho(F)$. 
    Thus, for all $\sigma\in \S_n$, we have that $\rho(\sigma\cdot F)=\sigma\cdot \rho(F)$. Hence, $\rho$ is an $\S_n$-module isomorphism.

    By definition of $\rho$ and the definitions of area on forests and parking functions, we have that $\area(F)=\area(\rho(F))$ for any $F\in \F_n$. 
    Since $\area(\alpha)=\area(\sigma\cdot \alpha)$ for any $\alpha\in \PF_n$ and for any $\sigma\in\S_n$, it follows that $\area(F)=\area(\rho(F))=\area(\sigma\cdot \rho(F))=\area(\rho(\sigma\cdot F))=\area(\sigma\cdot F)$. 
    Hence, $\CC[\F_n]$ is a graded $\S_n$-module with grading given by area, and $\rho$ is a graded module isomorphism.
\end{proof}

   \begin{remark}This parental preorder inversion statistic is not the usual statistic on $\F_n$ that is related to the area on $\PF_n$. An \defterm{inversion} of $F\in \F_n$ is a pair $(i,j)$ with $1\leq i<j\leq n$ such that $j$ is an ancestor of $i$. Let $\inv(F)$ denote the number of inversions of $F$. 
    Kreweras \cite{kreweras1980famille} showed that $\inv$ on $\F_n$ is \textit{equidistributed} with $\area$ on $\PF_n$, meaning
    \[\sum_{\alpha\in \PF_n}t^{\area(\alpha)}=\sum_{F\in \F_n}t^{\inv(F)}.\]Hence, \Cref{isomorphism-Sn-module} shows that $\inv$ on $\F_n$ is equidistributed with $\area$ on $\F_n$ i.e.,
       \begin{equation}
        \sum_{F\in \F_n}t^{\area(F)}=\sum_{F\in \F_n}t^{\inv(F)}.
        \end{equation}
        To prove this result bijectively, one could compose $\rho$ with the bijection of Kreweras. It would be of interest to have a direct bijection and we leave this an open problem.
    \end{remark}

Henceforth, we use $\mathbf{x}$ to denote the infinite sequence of variables $x_1,x_2,\dots$. 
Let $\PF_n(\mathbf{x})$ be the Frobenius image of $\CC[\PF_n]$, or equivalently (by \Cref{isomorphism-Sn-module}) of $\CC[\mathcal{F}_n]$. 
This is called the \defterm{parking function symmetric function} and was first introduced by Haiman \cite{haiman_conjectures_1994}. 
Define the \defterm{parental content} of $F\in \F_n$ to be the vector $\con(F)=(c_1,c_2,\dots, c_{n})$, where $c_i$ is the number of children of $w_F(i)$. 
It is straightforward to verify that $\con(F)=\con(\rho(F))$. 
Let $\F_n^\uparrow$ be the set of rooted labeled forests with no parental preorder inversions.
Then, we have 
\begin{equation}\label{eq:PF-sym-as-rlf}
    \PF_n(\mathbf{x})=\sum_{\alpha\in \PF_n^\uparrow}h_{\con(\alpha)}=\sum_{F\in \F_n^\uparrow}h_{\con(F)}.
\end{equation}
Haiman provides a generating function formula for $\PF_n(\mathbf{x})$. 
To state it, begin by recalling the following generating functions:

\begin{align}
\label{def: H(z)}
H(z)&=\sum_{n\geq 0}h_nz^n=\prod_{i\geq 1}\frac{1}{1-x_iz}
\qquad \mbox{and} \qquad E(z)=\sum_{n\geq 0}e_nz^n=\prod_{i\geq 1}(1+x_iz) \eqpd 
\end{align}
Note that $H(z)E(-z)=1$. Next, if $G(z)$ is a formal power series in $z$ (with coefficients in some commutative ring) with no constant term,  let $P(z)^{\langle -1\rangle}$ denote the \defterm{compositional inverse} of $G$, i.e., a formal power series $Q(z)$ such that $Q(P(z))=z$. 

Haiman's 1994 paper on diagonal coinvariant rings includes the conjecture of the existence of some module that was isomorphic to the parking function module \cite[p. 28]{haiman_conjectures_1994}. Stanley later proved that the equality $\PF(\mathbf{x},z)=(zE(-z))^{\langle -1\rangle }$ can be derived directly from Haiman's work using Lagrange inversion \cite[Proposition 2.2]{stanley1997parking}. We provide a combinatorial proof of of this equality 
via the map $\rho$, which we defined in \eqref{def:rho}.

 \begin{theorem}[{\cite[Proposition 2.2]{stanley1997parking}}]\label{thm:lambet-w-symmetric}  We have
    $\PF(\mathbf{x},z)\coloneqq \sum_{n\geq 0}\PF_n(\mathbf{x})z^{n+1}=(zE(-z))^{\langle -1\rangle }$.
\end{theorem}

\begin{proof}
By definition, the claim is equivalent to showing that 
\[\PF(\mathbf{x},z)E(-\PF(\mathbf{x},z))=z.\]

    Since $H(z)E(-z)=1$, the claim is equivalent to showing that 
    \begin{equation}\label{eq:symmetric-function-recurrsion}\sum_{n\geq 0}\PF_n(\mathbf{x})z^{n}=H(\PF(\mathbf{x},z)) \eqpd
    \end{equation}
    The right-hand side of \eqref{eq:symmetric-function-recurrsion} becomes
    \begin{align*}
        H(\PF(\mathbf{x},z))&=\sum_{k\geq 0} h_k \PF(\mathbf{x},z)^k\\    &=\sum_{k\geq 0} h_k\lp \sum_{n\geq 0}\PF_n(\mathbf{x})z^{n+1}\rp^k\\
        &=\sum_{n\geq 0}z^n \sum_{k=0}^n h_k \sum_{c\in \Comp(n,k)}\prod_{i=1}^k \PF_{c_{i-1}}(\mathbf{x})\\
        &=\sum_{n\geq 0}z^n \sum_{k=0}^n h_k \sum_{c\in \WComp(n-k,k)}\prod_{i=1}^k \PF_{c_{i}}(\mathbf{x}),
    \end{align*}
    where $\Comp(n,k)$ is the set of compositions of $n$ with $k$ parts and $\WComp(n,k)$ is the set of weak compositions of $n$ with $k$ parts.  Thus, comparing the coefficients of $z^n$ in \Cref{eq:symmetric-function-recurrsion}, the statement we now need to prove is 
    \begin{equation}\label{eq:tree-recurrsion}
        \PF_n(\mathbf{x})=\sum_{k=0}^n h_k \sum_{c\in \WComp(n-k,k)}\prod_{i=1}^k \PF_{c_{i}}(\mathbf{x}).
    \end{equation}

    To prove this, we appeal to the fact that elements of  $\mathcal{F}^{\uparrow}_n$ can be obtained uniquely by labeling unlabeled plane forests on $n$ vertices, as noted after \Cref{prop:parental-preorder-inversion-gf} (see also \Cref{fig:unlabled-plane-rooted-forest}). Let $\mathcal{U}_n$ be the set of unlabeled plane forest on $n$ vertices. We define $\con(U)=\con(F)$ if $F\in \mathcal{F}^{\uparrow}_n$ is the unique way to label $U$ to obtain a rooted labeled forest with no preorder traversals. Then, we have
    \begin{equation}\label{eq:PF-to-UF}\PF_n(\mathbf{x})=\sum_{U\in \mathcal{U}_n} h_{\con(U)}.\end{equation}

    We can build an element $U\in \mathcal{U}_n$ as follows:
    \begin{itemize}
        \item Select an integer $k$ and a weak composition $c\in \WComp(n-k,k)$.
        \item For each $i\in [k]$, pick some element $U_i\in \mathcal{U}_{c_i}$, and form a tree by attaching the roots of $U_i$ to a new root vertex. 
        The forest $U$ is the union of these trees.
    \end{itemize} 
    Let  $F=(F_1,\dots,F_k)\in \mathcal{F}_n^{\uparrow}$ be the unique labeling of $U=(U_1,\dots, U_k)$ that has no parental preorder inversions. We can compute the parental content $c=(c_1,\dots,c_n)$
    of $F$ as follows: 
    \begin{enumerate}
        \item Having $k$ roots means you have $c_1=k$.
        \item  For each $i\in [k]$ and $j\in F_i$, $F_i$ contributes $\con(F_i)_j$ to $c_j$. Since the labels on the vertices of $F_1,F_2,\dots,F_k$ form a partition of $[n]$, we have that $c_j=\con(F_i)_j$. 
    \end{enumerate}
    In other words, thinking of $\con(F), \con(F_1),\con(F_2),\dots,\con(F_k)$ as vectors of  length $n$, we have
        \[\con(F)=(k,\underbrace{0,0,\dots,0}_{n-1})+\con(F_1)+\con(F_2)+\cdots + \con(F_k),\]
        which implies that 
        \begin{equation}h_{\con(F)}=h_k h_{\con(F_1)}h_{\con(F_2)}\cdots h_{\con(F_k)}.\end{equation}
        Therefore, 
        \begin{equation}\label{eq:PF-to-UF-recur}\PF_n(\mathbf{x})=\sum_{F\in \F^{\uparrow}_n} h_{\con(F)}=\sum_{k=0}^n h_{k}\sum_{(U_1,\dots,U_k)}\prod_{i=1}^k {h}_{\con(U_i)},\end{equation}
        where the inner sum in the rightmost expression is over all $k$-tuples of unlabelled rooted plane forests on a total $n-k$ vertices. 
        To create such a $k$-tuple, we can first specify a weak composition $c\in \WComp(n-k,k)$ and then choose $U_i\in \mathcal{U}_{c_i}$ for each $i\in [k].$ 
        We then have 
        \begin{equation}\label{eq:UF-recurrsion}\sum_{(U_1,\dots,U_k)}\prod_{i=1}^k {h}_{\con(U_i)}=\sum_{c\in \WComp(n-k,k)}\prod_{i=1}^k\sum_{U_i\in \mathcal{U}_{c_i}}{h}_{\con(U_i)}
   =\sum_{c\in \WComp(n-k,k)}\prod_{i=1}^k \PF_{c_{i}}(\mathbf{x}).\end{equation}
        Then \Cref{eq:tree-recurrsion}  follows from combining Equations \eqref{eq:PF-to-UF}, \eqref{eq:PF-to-UF-recur}, and \eqref{eq:UF-recurrsion}.
\end{proof}

\subsection{Generating function formula for inversions on parking functions}
The next ingredient in understanding inversions for parking functions is a general statement about words. 
For this, recall that the \defterm{stable principal specialization} of a symmetric function $f$ is defined to be 
\[
\ps(f)=f(1,q,q^2,\dots)\eqpd
\]
 In other words, we set $x_i$ to $q^{i-1}$ for all $i\in \N$.

Next, for all nonnegative integers $n$, define 
\[(q;q)_n\coloneqq (1-q)(1-q^2)\cdots(1-q^n)=(1-q)^n[n]_q! \eqpd\]
 A set of length $n$ words $W\subseteq \N^n$ is \defterm{$\S_n$-invariant} if for all $w\in W$ and $\sigma\in \S_n$, 
\[\sigma\cdot w\coloneqq w_{\sigma^{-1}(1)}w_{\sigma^{-1}(2)}\cdots w_{\sigma^{-1}(n)}\in W.\]
In other words, the $\S_n$-action $\sigma\cdot w$ defined immediately above makes $\CC[W]$ an $\S_n$-module. There is a standard connection between stable principal specialization and inversion enumerators for words.
\begin{proposition}\label{lem:character-of-words-inversions}
    Let $W\subseteq \N^n$ be an $\S_n$-invariant set of words.
    Let $F(\mathbf{x})$ be the Frobenius image of the $\S_n$-module $\CC[W]$. Then, 
    \[(q;q)_n\ps(F(\mathbf{x}))=\sum_{w\in W}q^{\inv(w)}.\]
\end{proposition}

\begin{proof}
It is known \cite[Equation 9.4]{shareshian_chromatic_2016}, \cite[Lemma 5.2]{GESSEL1993189} that for any composition $\alpha$,
\[(q;q)_n\ps(h_\alpha)=\qbinom{n}{\alpha}_q.\]
Recall that if $w\in \N^n$ has content $\alpha$ 
and $\S_n\cdot w$ denotes the $\S_n$-orbit, then MacMahon \cite{MacMahon} showed that
    \begin{equation*}\sum_{u\in \S_n\cdot w}q^{\inv(u)}=\qbinom{n}{\alpha}_q.\end{equation*}
    Since the Frobenious image of $\CC[\S_n\cdot w]$ is $h_{\alpha}$, the claim follows.
\end{proof}

\begin{corollary}\label{cor:sym-to-q-exp-gen-funs}
    Let $(W_n)_{n\geq 1}$ denote a family of words such that $W_n$ is $\S_n$-invariant for each $n$. Also for each $n$, let $F_n(q)=\sum_{w\in W_n}q^{\inv(w)}$ be the inversion generating function for $W_n$ and let $F_n(\mathbf{x})$ denote the Frobenius image of $W_n$. Then, we have
    \[\sum_{n\geq 1}\ps(F_n(\mathbf{x}))z^n(1-q)^n=\sum_{n\geq 1}F_n(q)\frac{z^n}{[n]_q!} \eqpd
    \]
\end{corollary}
 In other words, first applying stable principal specialization and then substituting $z\mapsto z(1-q)$ turns the generating function for the Frobenius image into the $q$-exponential generating function for the inversion enumerators.
\begin{proof} 
For $n\geq 1$, we have 
    \[\ps(F_n(\mathbf{x}))(z(1-q))^n =z^n (1-q)^n \sum_{w\in W^{\uparrow}}\ps(h_{\con(w)})=z^n(1-q)^n \frac{F_n(q)}{(q;q)_n}=F_n(q)\frac{z^n}{[n]_q!}, 
    \]
where $W^\uparrow$ denotes the set of weakly increasing words $W$.
\end{proof}
From this, we obtain a generating function for the inversion enumerator for parking functions. Recall two $q$-exponential generating functions (see; \cite[Equations 1.85, 1.87]{stanley2012enumerative}): 
\[
\exp_q(z)=\sum_{n\geq 0}\frac{z^n}{[n]_q!}=\prod_{i\geq 1}\frac{1}{1-z(1-q)q^i}\quad\text{and}\quad \Exp_q(z)=\sum_{n\geq 0}q^{\binom{n}{2}}\frac{z^n}{[n]_q!}=\prod_{i\geq 1}1+z(1-q)q^i\eqpd
\]
One can see that $\exp_q(z)\Exp_q(-z)=1$ using the product formula.

\begin{corollary} \label{thm:PF-is-q-LambertW-Inverse} We have 
\[\sum_{n\geq 0}\PF_n(q)\frac{z^{n+1}}{[n]_{q}!}=(z\Exp_{q}(-z))^{\langle -1\rangle }.\]
\end{corollary}
We provide two proofs.
\begin{proof}[Proof 1: Specialization]
    From \cite[Proposition 7.8.3]{stanley1999enumerative} we have
    \[\ps(E(z(1-q)))=\sum_{n\geq 0} \ps(e_n) z^n(1-q)^n=\sum_{n\geq 0}\frac{q^{\binom{n}{2}}}{(q;q)_n}z^n (1-q)^n=\sum_{n\geq 0}q^{\binom{n}{2}}\frac{z^n}{[n]_q!}= \Exp_q(z) \eqpd
    \]
    Thus, the claim follows from \Cref{thm:lambet-w-symmetric} and \Cref{cor:sym-to-q-exp-gen-funs}.
\end{proof}
The proof of \Cref{thm:PF-is-q-LambertW-Inverse} implies that $\ps(H(z(1-q))=\exp_q(z)$.
Next, we provide a more elementary explanation of \Cref{thm:PF-is-q-LambertW-Inverse} that does not rely on symmetric functions through rooted labeled forests.

     \begin{proof}[Proof 2: Combinatorial]
     By definition, $\PF(z,q)$ is the compositional inverse of $z\Exp_q(-z)$ if and only if 
    \begin{align}
        \PF(z,q)\Exp_q(-\PF(z,q))=z,
    \end{align}
    which is equivalent to 
    \begin{equation}\sum_{n\geq 0}\PF_n(q)\frac{z^{n}}{[n]_q!}=\exp_q(\PF(z,q)) \eqpd \end{equation}
    The right-hand side can be expanded directly as
    \begin{align*}
        \exp_q(\PF(z,q))&=\sum_{k\geq 0}\frac{1}{[k]_q!}\lp \PF(z,q)\rp^k\\
        &=\sum_{k\geq 0}\frac{1}{[k]_q!}\lp \sum_{n\geq 0}\PF_n(q)\frac{z^{n+1}}{[n]_q!}\rp^k\\
        &=\sum_{k\geq 0}\frac{z^k}{[k]_q!} \sum_{n\geq 0}z^n\sum_{c=(c_1,\ldots,c_k)\in \Comp(n,k)}\prod_{i=1}^k\frac{\PF_{c_i-1}(q)}{[c_i]_q!}\\
         &= \sum_{n\geq 0}z^n\sum_{k=0}^n\frac{1}{[k]_q!}\sum_{c=(c_1,\ldots,c_k)\in \WComp(n-k,k)}\prod_{i=1}^k\frac{\PF_{c_i}(q)}{[c_i]_q!}\\
         &= \sum_{n\geq 0}\frac{z^n}{[n]_q!}\sum_{k=0}^n\frac{[n]_q}{[n-k]_q![k]_q!}\sum_{c=(c_1,\ldots,c_k)\in \WComp(n-k,k)}\frac{[n-k]_q!}{[c_1]_q![c_2]_q!\cdots [c_k]_q! }\prod_{i=1}^k\PF_{c_i}(q)\\
         &=\sum_{n\geq 0} \frac{z^n}{[n]_q!}\sum_{k=0}^n \sum_{c=(c_1,\ldots,c_k)\in \WComp(n-k,k)}\qbinom{n}{k}_q\qbinom{n-k}{c_1,c_2,\dots, c_k}_q  \prod_{i=1}^k \PF_{c_i}(q),
    \end{align*} 
    following standard algebraic manipulations. 
    
    Hence, we want to show that
    \begin{equation}\label{eq:PF-inversion-decomposition}\PF_n(q)=\sum_{k=0}^n\sum_{c=(c_1,\ldots,c_k)\in \WComp(n-k,k)}\qbinom{n}{k}_q\qbinom{n-k}{c_1,c_2,\dots, c_k}_q  \prod_{i=1}^k \PF_{c_i}(q).\end{equation}
      Now we can prove the equivalent statement of \eqref{eq:PF-inversion-decomposition} for rooted labeled forests. We can construct a rooted labeled forest $F$ as follows:
    \begin{enumerate}
        \item Pick a subset $R=\{r_1<r_2<\cdots<r_k\}$ of  $k$ numbers of $[n]$ to be the roots of your trees. 
        \item Pick a weak composition $c=(c_1,\ldots,c_k)$ of $n-k$ of size $k$, and select an ordered set partition $B_1,B_2,\dots,B_{k}$ of $[n]\setminus R$ with $|B_i|=c_i$ for each $i\in[k]$.
        \item Pick elements $F_{i}\in \mathcal{F}_{B_i}$ (i.e., a rooted forest labeled with $B_i$) and attach their roots to $r_i$ for each $i\in [k]$.
    \end{enumerate}
    
    To count the number of parental preorder inversions of $F$, we can see that parental preorder inversions $(i,j)$ come from one of three places:
    \begin{itemize}
        \item $i\in B_k$ and $j\in R$ (the parent of the elements of $R$ always comes first in the preorder transversal),
        \item $i\in B_k$, $j\in B_\ell$, and $k>\ell$ ($B_k$ is traversed after $B_\ell$ if $k>\ell$), or
        \item $i,j$ are both in some $B_\ell$.
    \end{itemize}
    These correspond to the 3 factors in the outer sum of \Cref{eq:PF-inversion-decomposition}.
    Thus, we have shown that
    \begin{equation}\label{eq:RLF-decomposition}
        \sum_{F\in \mathcal{F}_n}q^{\pinv(F)}=\sum_{c\in \WComp(n-k,k)}\qbinom{n}{k}_q\qbinom{n-k}{c_1,c_2,\dots, c_k}_q  \prod_{\ell=1}^k \sum_{F\in \mathcal{F}_{c_\ell}}q^{\pinv(F)} \eqcom
    \end{equation}
    and so the claim follows.
\end{proof}

\section{Inversions in unit interval parking functions}\label{sec:inversions-in-upfs}

Let $\alpha = (a_1, a_2, \ldots , a_n)\in \PF_n$.
As in \cite{bradt2024unit}, we call $\alpha$ a \textbf{unit interval parking function} if every car parks in its preference or in the spot after its preference and write $\alpha \in \UPF_n$.
In \Cref{subsec:cayley biject to upf}, we begin our study by providing a bijection between the set of unit interval parking functions and the set of Cayley permutations, which have appeared in the literature under various names.
In \Cref{subsec:inv gen fuct  for UPF}, we give a generating function for the number of inversions in unit interval parking functions.
In \Cref{subsec:sym fct upf results}, we conclude with some related  symmetric function results.

\subsection{A connection to Cayley permutations }\label{subsec:cayley biject to upf}
We begin by providing a connection between Cayley permutations and unit interval parking functions.
A \emph{Cayley permutation} of length $n$ is a word $w\in \N^n$ such that if $i$~appears in $w$, then each $j\in [i]$ also appear in $w$. 
We denote the set of all Cayley permutations of length $n$ by $\mathfrak{C}_n$.  
Cayley permutations of length $n$ can be equivalently characterized as the set $\C_n = \{w\in [n]^n : \con(w) \in\Comp(n) \}$.
Cayley permutations appear in the literature under various names, including \emph{Fubini words}, \textit{packed words}, \textit{surjective words}, and \textit{initial words} \cite{Marberg_2021, caylerian}. 
We note that Cayley permutations have also been viewed as ordered set partitions, and descent generating functions for Cayley permutations have been studied in~\cite{cerbai2024caylerianpolynomials}. Refined counts of ordered set partitions have been studied before but in different contexts; see \Cref{remark:OSP-invs} for further details. 

First, we need an important characterization of unit interval parking functions. 
Let $\beta$ be the weakly increasing rearrangement of $\alpha$ and suppose $\{i\in [n]\mid \beta_i=i\}=\{i_1 < i_2 < \cdots < i_k\}$.  
Then, the \defterm{block structure} \cite[Definition 2.7]{bradt2024unit} of $\alpha$ is a decomposition of $\alpha$ into blocks $\pi_1 \mid \pi_2\mid\cdots \mid \pi_{k}$ of strings defined by
\[\pi_j=\beta_{i_j}\beta_{i_j+1}\cdots \beta_{i_{j+1}-1}\]
for each $j\in [k]$. 
This is also called the \textbf{prime decomposition} of $\alpha$ in \cite[Remark 2.7]{unit_perm}.

The following result establishes that a unit interval parking function is uniquely determined by its block sizes and the positions of its blocks.

\begin{theorem}[{\cite[Theorem.~2.9]{bradt2024unit}}]
\label{thm:upf_rearrangement}
Let $\alpha = (a_1,a_2, \dots, a_n) \in \UPF_{n}$ and let $\alpha^\uparrow = (a'_1,a_2', \dots, a'_n)$ be the weakly increasing rearrangement of $\alpha$.
Then:
\begin{enumerate}
\item $a'_i\in\{i,i+1\}$ for all $i$.
\item Place a separator in $\alpha^\uparrow$ before each element $a'_i$ satisfying $a'_i=i$, we obtain a decomposition $\alpha^\uparrow=\pi_1\mid\pi_2\mid\cdots\mid\pi_m$, where each $\pi_j$ is of the form $(i,i,i+1,i+2,\dots,i+j)$.  The words $\pi_j$ are called the \defterm{blocks} of $\alpha$.
\item The sizes of the blocks determine $\alpha^\uparrow$ uniquely.
\item The entries in each block appear in increasing order in $\alpha$.  Equivalently, a rearrangement $\sigma$ of $\alpha$ is in $\UPF_{n}$ if and only if the entries in $\sigma$ respect the relative order of the entries in each block $\pi_j$.
\item In particular, the number of UPFs with increasing rearrangement $\alpha^\uparrow$ is $\binom{n}{|\pi_1|,|\pi_2|,\ldots, |\pi_m|}$.
\end{enumerate}
\end{theorem}
Next, we utilize \Cref{thm:upf_rearrangement} to establish a bijection between unit interval parking functions and Cayley permutations. 

\begin{lemma}\label{thm:bijection upf and cayley}
    Define $\psi:\UPF_n\to \C_n$ as follows: suppose $\alpha\in \UPF_n$ has block-structure $\pi_1\mid\pi_2\mid\cdots\mid\pi_k$. 
    Let $\alpha_i$ denote the $i$th entry of $\alpha$, and $\psi(\alpha)_i$ denote the $i$th entry of $\psi(\alpha)$.
    Set
    \[\psi(\alpha)_i=j\quad \text{ if $\alpha_i$ appears in $\pi_j$}\]
    for each $i\in n$. Then $\psi$ is a bijection such that for all $\alpha\in \UPF_n$, $\Inv(\psi(\alpha))=\Inv(\alpha)$.
\end{lemma}
\begin{proof}
    First, the image of $\psi$ is indeed a Cayley permutation since the numbers $1,2,\dots,k$ appear at least once each in $\psi(\alpha)$ if $\alpha$ has $k$ blocks.

    Suppose $\alpha\in \UPF_n$ and $\alpha_i>\alpha_j$ for some $i<j$. By \Cref{thm:upf_rearrangement} part (4),  $\alpha_i$ is in a strictly larger (indexed) block than $\alpha_j$. Hence, $\psi(\alpha)_i>\psi(\alpha)_j$. On the other hand, if $\alpha_i\leq \alpha_j$ for some $i<j$, then $\alpha_i$ and $\alpha_j$ are either in the same block, or $\alpha_i$ is in a smaller (indexed) block than $\alpha_j$. Hence, $\psi(\alpha)_i\leq \psi(\alpha)_j$. Thus, $\Inv(\psi(\alpha))=\Inv(\alpha)$. 

    Suppose $\psi(\alpha)=\psi(\beta)$. Then, we can see that $\alpha$ and $\beta$ must have the same block structure since the contents of $\psi(\alpha)$ and $\psi(\beta)$ are the same. By definition, $\psi(\alpha)_i=\psi(\beta)_i=j$ if an element of block $\pi_j$ appears in position $i$ of $\alpha$ and $\beta$. By \Cref{thm:upf_rearrangement} part (4), there is only one way to arrange the elements of each block $\pi_j$ such that an element of $\pi_j$ appears in the positions $i$ such that $\psi(\alpha)_i=j$. Hence, $\alpha=\beta$. 

    Given a Cayley permutation $w\in \C_n$, let $c=\con(w)=(c_1,c_2,\dots,c_k)$ be the content of $w$. 
    For each $j\in [k]$, if $i_1<i_2<\cdots<i_{c_j}$ are the positions of the $j$'s in $w$, then replace the $j's$ (from left to right) by $a,a,a+1,a+2,\dots,a+c_j-1$, where $a=c_1+c_2+\cdots+c_{j-1}+1$ to construct a sequence $\alpha$ (where we set $c_0=0$). 
    The sequence $\alpha$ is proven to be a unit interval parking function in the proof of \cite[Theorem 2.5]{bradt2024unit}. 
    Moreover, $\psi(\alpha)=w$ since the block structure of $\alpha$ is $c$.
\end{proof}

\subsection{Inversion generating function for unit interval parking functions}\label{subsec:inv gen fuct  for UPF}
For $\pi\in \S_n$, let $\asc(\pi)$ denote the number of \defterm{ascents} of $\pi=\pi_1\pi_2\cdots\pi_n$ i.e., the number of positions $i\in [n-1]$ such that $\pi_i<\pi_{i+1}$. 
Define for each $n\in \N$, the polynomials
\begin{equation}
        A_n^{\inv,\asc}(q,t)=\sum_{\pi \in \S_n} q^{\inv(\pi)}t^{\asc(\pi)}\qquad \mbox{and}\qquad
        \UPF_n(q)= \sum_{\alpha \in \UPF_n} q^{\inv(\alpha)}.
\end{equation}
We provide two methods for obtaining a $q$-exponential generating function for $\UPF_n(q)$: (1) by showing that $\UPF_n(q)$ is a specialization of $A_n^{\inv,\asc}(q,t)$ and applying a theorem of Stanley \cite{stanley1976binomial}; and (2) a direct computation. We start with the former.
\begin{proposition}\label{Theorem: inv from Sn}
     For $n\geq 1$,  $A_n^{\inv,\asc}(q,2)=\UPF_n(q)$.
\end{proposition}
\begin{proof}
    By \Cref{thm:bijection upf and cayley}, we have 
    \[\UPF_n(q)=\sum_{w\in \C_n}q^{\inv(w)}.\]
   Let $\Asc(\sigma)=\{i\in[n-1]:\sigma_i<\sigma_{i+1}\}$  denote the set of ascents of a permutation $\sigma$. Define a set \[\mathcal{P}_n=\{(\sigma,S)\mid \sigma\in \S_n,\ S\subseteq \Asc(\sigma^{-1})\}.\]
    Define a map $\eta:\mathcal{P}_n\to \C_n$ as follows: suppose $[n]\setminus S=\{s_1<\cdots<s_r\}$. Define a partition $B=(B_1,\dots,B_{r})$ of $[n]$ by 
    \[B_k=\{s_{k-1}+1,s_{k-1}+2,\dots,s_{k}\}\]
    for each $k\in [r]$, where $s_0=0$. For example, if $(\sigma,S)=(236841759,\{2,3,6,8\})$, then $\sigma^{-1}=612583749$ and $\Asc(\sigma^{-1})=\{2,3,4,6,8\}$. Note $S\subseteq \Asc(\sigma^{-1})$. So $[9]\setminus S=\{s_1=1, s_2=4, s_3=5,s_4=7,s_5=9\}$, and 
    \[B=(\{1\},\{2,3,4\},\{5\},\{6,7\},\{8,9\}).\]
    Then, set $w_i=k$ if $\sigma_i\in B_k$ and set $\eta(\sigma,S)=w=w_1w_2\cdots w_n$. Continuing the example, \[\eta(236841759,\{2,3,6,8\})=224521435.\] Since each block of the partition is nonempty, $w\in \C_n$. Suppose $(i,j)\in \Inv(\sigma)$. Then $\sigma_i$ must be in a larger indexed block than the one containing $\sigma_j$ by construction. Hence, $w_i>w_j$, so $(i,j)\in \Inv(w)$. 
    
    On the other hand, suppose we are given $w\in \C_n$ with $r$ distinct values and set $B_k=\{i\in [n]\mid w_i=k\}$ for each $k\in [r]$. 

Suppose the content of $w$ is $c=(c_1,\dots,c_r)$. Then for each $k\in [r]$, set $s_k=c_1+\cdots+c_{k}$ and then define 
\[B_k=\{s_{k-1}+1,s_{k-1}+2,\dots,s_k\},\]
where $s_0=0$. Then we set $\sigma$ to be obtained by replacing the $j$-th leftmost $k$ in $w$ by $s_{k-1}+j$ for each $k\in [r]$ and $j\in [c_k]$.
For example, let $w=224 521 435$. The  content of $w$ is $c=(1,3,1,2,2)$, so $s_1=1,s_2=4,s_3=5,s_4=7,s_5=9$  and $(B_1,\dots,B_5)=(\{1\},\{2,3,4\},\{5\},\{6,7\},\{8,9\}).$ Moreover, $\sigma=236 841 759$.  Now, set 
\[S\coloneqq \bigcup_{k=1}^r B_k\setminus \{\max(B_k)\}.\]
Suppose $m\in S$, say $m\in B_k\setminus \max(B_k)$. Then, $m+1\in B_{k}$ because $B_k$ contains only consecutive numbers. Hence, $m=s_{k-1}+j$ and $m+1=s_{k-1}+j+1$ for some $j\in [c_k-1]$. By construction of $\sigma$, $\sigma_i=m$ implies that $w_i$ is the $j$-th leftmost occurrence of $k$ in $w$, and moreover the $(j+1)$-th occurrence of $k$ is at some positions $w_{i'}$ with $i'>i$. Then, $\sigma_{i'}=m+1$, so $m\in \Asc(\sigma^{-1})$. Therefore, $S\subseteq \Asc(\sigma^{-1}).$

    From here, one can show that $\eta(\sigma,S)=w$, so $\eta$ is injective. Since 
    \[|\C_n|=\Fub_n=\sum_{\sigma\in \S_n}2^{\asc(\sigma)},\]
    (see \cite[p.~269, eqn.~(6.39)]{GKP} for second equality), $\eta$ is a bijection. Putting this all together, we have 
    \begin{align*}
    \UPF_n(q)&=\sum_{w\in \C_n}q^{\inv(w)}\\
    &=\sum_{(\sigma,S)\in \mathcal{P}_n}q^{\inv(\sigma)}\\
    &=\sum_{\sigma\in \mathfrak{S}_n}q^{\inv(\sigma)}\cdot |\{S\subseteq[n]\mid S\subseteq \Asc(\sigma^{-1})\}|\\
    &=\sum_{\sigma\in \mathfrak{S}_n}q^{\inv(\sigma)}2^{\asc(\sigma^{-1})}\\
    &=\sum_{\sigma\in \mathfrak{S}_n}q^{\inv(\sigma)}2^{\asc(\sigma)}\\
    &=A_n^{\inv,\asc}(q,2),
    \end{align*}
    as desired.
\end{proof}

We provide two proofs of the following result.
\begin{theorem}\label{eq:upf-inv-gen-func}
    We have
      \[\sum_{n\geq 0} \UPF_{n}(q)\frac{z^n}{[n]_q!}=\frac{1}{2-\exp_{q}(z)} \ .\]
\end{theorem}
\begin{proof}[Proof 1: Specialization]
Stanley \cite[Corollary 3.6]{stanley1976binomial} proved that $A_n^{\inv,\asc}(q,t)$ satisfies
\begin{equation}\label{eq:inv-asc-gen-function}
    1+\sum_{n\geq 1}t A_{n}^{\inv,\asc}(q,t)\frac{z^n}{[n]_q!}=\frac{1-t}{1-t\Exp_q(z(1-t))} \ .
\end{equation}

    Set $t=2$ in \eqref{eq:inv-asc-gen-function} 
    to get

   \begin{align*}
        \sum_{n\geq 1}2\UPF_n(q)\frac{z^n}{[n]_q!}=\sum_{n\geq 1}2A_{n}^{\inv,\asc}(q,2)\frac{z^n}{[n]_q!}=\frac{-1}{1-2\Exp_q(-z)}+1=\frac{\exp_q(z)}{2-\exp_q(z)}+1
    \end{align*}

where the first equality comes from \Cref{Theorem: inv from Sn}, while the third follows from the identity $\exp_q(-z)=1/\Exp_q(z)$. From there, the proof follows by algebraic manipulations.
\end{proof}

\begin{proof}[Proof 2: Elementary]
This result can be found directly without relying on \Cref{eq:inv-asc-gen-function}. Indeed, \Cref{thm:bijection upf and cayley} implies that 
\begin{equation}\UPF_n(q)=\sum_{w\in \C_n}q^{\inv(w)}=\sum_{c\in \Comp(n)} \sum_{\substack{w\in [n]^n\\\con(w)=c}}q^{\inv(w)}=\sum_{c\in \Comp(n)} \qbinom{n}{c_1,c_2,\dots,c_{\ell(c)}}_q,\end{equation}
where the final equality comes from MacMahon \cite{MacMahon}.
Starting with the right-hand side of \Cref{eq:upf-inv-gen-func}, we have 

\begin{align*}
    \frac{1}{2-\exp_{q}(z)}&=
    \frac{1}{2}\cdot \frac{1}{1-\frac{\exp_{q}(z)}{2}}\\
    &=\frac{1}{2}\sum_{k\geq 0}\frac{1}{2^k}\exp_q(z)^{k}\\
    &=\sum_{k\geq 0}\frac{1}{2^{k+1}}\lp \sum_{n\geq 0}\frac{z^n}{[n]_q!}\rp ^{k}\\
    &=\sum_{k\geq 0}\frac{1}{2^{k+1}}\sum_{n\geq 0}z^n \sum_{c\in \WComp(n,k)}\prod_{i=1}^{k}\frac{1}{[c_i]_q!}\\
    &=\sum_{k\geq 0}\frac{1}{2^{k+1}}\sum_{n\geq 0}\frac{z^n}{[n]_q!} \sum_{c\in \WComp(n,k)}\frac{[n]_q!}{[c_1]_q!\cdots [c_k]_q!}\\
    &=\sum_{k\geq 0}\frac{1}{2^{k+1}}\sum_{n\geq 0}\frac{z^n}{[n]_q!} \sum_{c\in \WComp(n,k)} \qbinom{n}{c_1,c_2,\dots,c_{k}}_q\\
    &=\sum_{n\geq 0}\frac{z^n}{[n]_q!} \sum_{k\geq 0}\frac{1}{2^{k+1}}\sum_{c\in \WComp(n,k)} \qbinom{n}{c_1,c_2,\dots,c_{k}}_q\\
     &=\sum_{n\geq 0}\frac{z^n}{[n]_q!} \sum_{c\in \WComp(n)} \qbinom{n}{c_1,c_2,\dots,c_{\ell(c)}}_q\cdot \frac{1}{2^{\ell(c)+1}},
\end{align*}
where we recall that $\ell(c)$ denotes the number of parts of a weak composition $c$.
Since every weak composition can be made by adding 0's to a (strong) composition, we have
\begin{align*}
    \sum_{c\in \WComp(n)} \qbinom{n}{c_1,c_2,\dots,c_{\ell(c)}}_q\cdot \frac{1}{2^{\ell(c)}}=\sum_{c\in \Comp(n)} \qbinom{n}{c_1,c_2,\dots,c_{\ell(c)}}_q \sum_{j\geq 0}\binom{\ell(c)+j}{j}\frac{1}{2^{\ell(c)+j}}.
\end{align*}

Expanding $ 2^{-(m+1)}(1-x)^{-(m+1)}$ with Newton's generalized binomial theorem and then substituting $x=1/2$, we have that for all positive integers $m$,
  \[\sum_{j\geq 0}\binom{m+j}{j}\frac{1}{2^{m+j+1}}=1.\]  
 Hence, the claim follows.
\end{proof}

\begin{remark}\label{remark:OSP-invs}
The definition of $\UPF_n(q)$ along with the closed form in \Cref{eq:upf-inv-gen-func}  make it a compelling $q$-analog of $\Fub_n$. However, there is a different $q$-analog of $\Fub_n$, which has received much more attention in the literature. Originally studied under the name $\operatorname{ROS}$ by Ste\'ingrimsson \cite{OSPstats}, an ``inversion,'' as defined by Remmel and Wilson \cite{RemmelWilson}, of a ordered set partition $B=(B_1,\dots,B_k)$ is pair $(i,B_j)$ such that $i$ appears in an earlier block than $B_j$ and $i$ is greater than something in $B_j$. This kind of inversion has been related to the Delta Conjecture, an important conjecture in the theory of symmetric functions \cite{DeltaConjecture}. This ``inversion'' statistic has been shown to be equidistributed with a number of other statistics, and the resulting $q$-analog has many wonderful properties \cite{DeltaConjecture,Rhoades,ospcoinv}. The Delta Conjecture context is also closely related to parking functions, leading us to wonder whether it is coincidence that ordered set partitions can be viewed as parking functions themselves. \end{remark}

\subsection{Unit interval parking function symmetric function}\label{subsec:sym fct upf results}
To generalize to a result on symmetric functions, we require an $\S_n$-action on $\UPF_n$. Suppose $\alpha\in \UPF_n$ has block structure $\pi_1\mid\pi_2\mid \cdots\mid \pi_k$, as described in \Cref{thm:upf_rearrangement}. 
For $i\in [n-1]$, define \[(i,i+1)\cdot \alpha=\begin{cases} \alpha&\text{if $\alpha_i,\alpha_{i+1}$ are in the same block}\\
(\alpha_1,\alpha_2,\dots,\alpha_{i-1},\alpha_{i+1},\alpha_i,\alpha_{i+2},\dots,\alpha_n)&\text{if $\alpha_i,\alpha_{i+1}$ are in different blocks.}
\end{cases}\]

One can check that this defines an $\S_n$-action on $\UPF_n$, as noted in \cite[pg.\ 11]{unit_perm}. 
The \defterm{content} of $\alpha$ is the composition \[\con(\alpha)=(|\pi_1|,|\pi_2|,\dots,|\pi_k|).\] 
In \cite[Theorem 2.9]{bradt2024unit}, the authors show that two unit interval parking functions have the same content if and only if they are rearrangements of each other. 
Moreover, they show that for each composition $c\in \Comp(n)$ there is a unit interval parking function $\alpha$ with $\con(\alpha)=c$.
Thus, \[\UPF_n(\mathbf{x})=\sum_{c\in \Comp(n)}h_c.\]
Since $\UPF_n$ and $\C_n$ have the same Frobenius image, they are isomorphic as $\S_n$-modules. 
In particular, we show that the map $\psi$, defined in \Cref{thm:bijection upf and cayley},
is an explicit isomorphism.
\begin{proposition}
   Let $\psi:\UPF_n\to \C_n$ be as in \Cref{thm:bijection upf and cayley}. 
   Then, for all $\alpha\in \UPF_n$, $\con(\psi(\alpha))=\con(\alpha)$. 
   Moreover, $\psi$ extends to an $\S_n$-module isomorphism.
\end{proposition}
\begin{proof}
    For $\alpha\in \UPF_n$ with block structure $\pi_1\mid\pi_2\mid \cdots \mid \pi_k$, The map $\psi$ replaces each element of block $\pi_j$ with a $j$ for each $j\in [k]$. Hence, entry $j$ of $\con(\alpha)=(|\pi_1|,|\pi_2|,\dots,|\pi_k|)$ is exactly the number of $j$'s in $\psi(\alpha)$ for each $j\in [k]$, so $\con(\alpha)=\con(\psi(\alpha))$. 

    Suppose $i\in [n-1]$. If $\alpha_i$ and $\alpha_{i+1}$ are in the same block, then $\psi(\alpha)_i=\psi(\alpha)_{i+1}$. Hence, $(i,i+1)\cdot \psi(\alpha)=\psi(\alpha)=\psi((i,i+1)\cdot \alpha)$. On the other hand, if $\alpha_i$ and $\alpha_{i+1}$ are in different blocks, then $\psi(\alpha)_i\neq\psi(\alpha)_{i+1}$, so
    \[(i,i+1)\cdot\psi(\alpha)=(\psi(\alpha)_1,\psi(\alpha)_2,\dots,\psi(\alpha)_{i-1},\psi(\alpha)_{i+1},\psi(\alpha)_i,\psi(\alpha)_{i+2},\dots,\psi(\alpha)_n)=\psi((i,i+1)\cdot \alpha).\]
Hence, $\psi$ is an isomorphism of $\S_n$-actions.
\end{proof}
\begin{theorem}\label{thm:UPF-sym-generating-function}
    We have
    $\sum_{n\geq 0}\UPF_n(\mathbf{x})z^n=\frac{1}{2-H(z)}$, where $H(z)$ is defined in \Cref{def: H(z)}.
\end{theorem}
  \begin{proof}The result follows by writing
  \[\frac{1}{2-H(z)}=\sum_{k\geq 0} \frac{1}
    {2^{k+1}}\lp \sum_{ n\geq 0}h_nz^n\rp ^{k}=\sum_{n\geq 0}z^n\sum_{c\in \Comp(n)}h_c\sum_{j\geq 0}\binom{\ell(c)+j}{j}\frac{1}{2^{\ell(c)+j+1}}
    \]
    and recalling that $\sum_{j\geq 0}\binom{\ell(c)+j}{j}\frac{1}{2^{\ell(c)+j+1}}=1$.
    \end{proof}

While setting $t=2$ makes sense for polynomials, it is not clear what the representation theoretic meaning is of this evaluation. In particular, we do not know how to derive \Cref{thm:UPF-sym-generating-function} from a known $\S_n$-representation on $\S_n$. 

Just like $\CC[\PF_n]$, $\CC[\UPF_n]$ is a graded $\S_n$-module, with the grading coming from the area function, the function
\[\UPF_n(\mathbf{x},t)\coloneqq\sum_{\alpha\in \UPF_n^\uparrow}t^{\area(\alpha)}h_{\con(\alpha)}\]
is the graded Frobenius image of $\CC[\UPF_n]$. This means that
\[\UPF_n(\mathbf{x},t)=\sum_{k=0}^{n-1} \ch\lp \CC\{\alpha\in \UPF_n\mid \area(\alpha)=k\}\rp t^k. \]
We can likewise get a nice expression for the generating function of $\UPF_n(\mathbf{x},t)$.

\begin{theorem}
The generating function of $\UPF_n(\mathbf{x},t)$ is
    \[\sum_{n\geq 0}\UPF_n(\mathbf{x},t) z^n=\frac{t}{[2]_t-H(tz)}\]
\end{theorem}
\begin{proof}
We first rewrite the right-hand side as \[\frac{t}{[2]_t-H(tz)}=\frac{t}{(t+1)-H(tz)}=\frac{1}{1-\frac{H(tz)-1}{t}}\]
and note that
\[H(tz)-1=\sum_{n\geq 1}h_nz^nt^n.\]
Then, we have \begin{align*}
    \frac{1}{1-\frac{H(tz)-1}{t}}&=\sum_{k\geq 0}\frac{1}{t^k}\lp \sum_{n\geq 1}h_n z^n t^n\rp ^k\\
    &=\sum_{n\geq 0}z^n \sum_{c\in \Comp(n)}t^{n-\ell(c)}  h_c. \end{align*} 

Now we just have to show that \[\UPF_n(\mathbf{x},t) = \sum_{c\in \Comp(n)}t^{n-\ell(c)}h_c.\]
Indeed, observe that if $\alpha\in \UPF_n^\uparrow$ and $\con(\alpha)=(c_1,c_2,\dots,c_{n-k})$, then for each $i\in [n-k]$, cars $2+\sum_{j=1}^{i-1}c_j, 3+\sum_{j=1}^{i-1}c_j, \dots, c_i+\sum_{j=1}^{i-1}c_j$ each park exactly 1 spot away from their preference, and thus contribute $c_i-1$ in total to the area. Thus, $\area(\alpha)=c_1+c_2+\dots+c_{n-k}-(n-k)=k$, and so we have the desired expression.
\end{proof}

\section{Expectations of \texorpdfstring{$k$}{k}-transitive statistics}\label{sec:expectations}

In this section, we continue our discussion on $\S_n$-invariant sets of words by connecting the total number of inversions on the set to the total number of descents in $\S_n$-invariant set of words. The parking function motivation comes from Schumacher's \cite[Theorem 10]{bib:DescentsInPF} which establishes 
\begin{equation}\label{eq: Schumacker descents}
    \sum_{\alpha\in \PF_n} \des(\alpha)=\binom{n}{2}(n+1)^{n-2}.
\end{equation}
Using probabilistic techniques,
we obtain an analogous formula for inversions on parking functions, which we state and prove in \Cref{thm:total-number-of-inversions-in-parking-functions}. Moreover, we extend this to a larger class of statistics on $\S_n$-invariant set of words, namely those defined via \textit{$k$-transitive functions}. 
 
For a random variable $f:\Omega \to \R$ on a finite set $\Omega$, define the \textbf{expectation} of $f$ with respect to the uniform measure on $\Omega$ to be \begin{equation}\label{def:expectation}\E_\Omega[f]=\frac{1}{|\Omega|}\sum_{x\in \Omega}f(x)=\sum_{y\in \R} \P_\Omega(f=y) \cdot y,
\end{equation}
where we denote the discrete uniform probability distribution by 
\[\P_\Omega(f=y)=\frac{|\{x\in \Omega\mid f(x)=y\}|}{|\Omega|}.\]
For further definitions and background in probability, see \cite{Durrett_2019,FellerProbability, ProbabilityGut}. When we know the size of a set and our random variables are chosen appropriately, computing the expectation of a statistic \eqref{def:expectation} is equivalent to computing the sum of the statistic over the elements of the set.
We use the language of expected value to clarify combinatorial arguments which would otherwise be much more cumbersome. Linearity of expectation and symmetry often lie at the heart of these arguments. In the same way that descents are inversions that happen next to each other, we can consider `adjacent' and `non-adjacent' versions of other statistics. In \Cref{thm:k-transitive-sn-invariant-inv-and-sum}, we prove a general result relating expectations (and therefore counts) of such pairs of statistics. Our first application is to obtain the total number of inversions of parking functions from Schumacher's count of the total number of descents. In \Cref{sec:applications-of-linearity-of-expectation}, we explore numerous other applications of this.

For positive integers $k\leq n$, define the set \[I_{n,k}=\{(i_1,i_2,\dots,i_k)\mid 1\leq i_1<i_2<\cdots<i_k\leq n\}.\] A function $\chi:\N^n\times I_{n,k}\to \R $ is said to be \defterm{$k$-transitive} provided there is a function $\theta:\N^k\to \R$ such that 
\begin{equation}
    \chi(w,(i_1,i_2\dots,i_k))=\theta (w_{i_1},w_{i_2},\dots,w_{i_k})
\end{equation}
for each $w\in \N^n$ and $(i_1,i_2,\dots,i_k)\in I_{n,k}$. In other words, $\chi$ depends only on the ordered list of values $w_{i_1},w_{i_2},\dots,w_{i_k}$ and not on the positions. The name comes from the fact that the symmetric group $\S_n$ \textbf{acts $k$-transitively} on $[n]$ for any $k\in [n]$, meaning if $\textbf{p}=(i_1,i_2,\dots,i_k)$ and $\textbf{q}=(j_1,j_2,\dots,j_k)$ are a vector of distinct values of $[n]$, then there is a permutation $\sigma\in \S_n$ such that $\sigma(i_s)=j_s$ for all $s\in [k]$. Moreover, if $\chi$ is $k$-transitive, then 
\begin{equation}\label{eq:k-transitive-invariance-of-stat}\chi(w, \textbf{p})=\chi(\sigma\cdot w,\textbf{q})\end{equation} 
for any $w\in \N^n$.

Let $W$ be a finite subset of $\N^n$ and let $\chi$ be a $k$-transitive function. Define random variables on $W$ by setting 
\begin{equation}\label{eq:f_and_g_sub_chi}
f_\chi(w)=\sum_{\textbf{p}\in I_{n,k}}\chi(w,\textbf{p})\quad \text{and}\quad g_\chi(w)=\sum_{i=0}^{n-k}\chi(w,(1+i,2+i,\dots,k+i))
\end{equation}
for all $w\in W$. 
We start by giving an example of a transitive function that both serves as a motivation as well as being of interest in what follows.
\begin{example}
Let $W$ be a finite subset of $\N^n$. For $w\in W$ and integers $1\leq i<j\leq n$, define
\[\chi^{\inv}(w,(i,j))=\begin{cases}
    1& w_i>w_j\\
    0&w_i\leq w_j.
\end{cases}\]
We can see that $\chi^{\inv}$ is a $2$-transitive function with $\theta(a,b)=1$ if $a>b$,  and $0$ if $a\leq b$. 
In the language of probability, $\chi^{\inv}$ is a classical example of an indicator random variable. 
Moreover, $f_{\chi^{\inv}}$ is the classical inversion statistic $\inv$, while $g_{\chi^{\inv}}$ is the classical descent statistic $\des$. 

For $W=\PF_n$, we can rephrase \eqref{eq: Schumacker descents} as a statement about the expectation $\E_{\PF_n}\left[\des\right] =\E_{\PF_n}\left[g_{\chi^{\inv}}\right]$  of the random variable $g_{\chi^{\inv}}$ on $\PF_n$. 
Analogously,  we want to compute the expectation $\E_{\PF_n}\left[\inv\right]$ of $\inv=f_{\chi^{\inv}}$. 
Note that the probabilistic interpretation of these computations is to determine the expected number of decents (or inversions) in a uniformly random parking function of length $n$.
\end{example}

For a $k$-transitive function $\chi$, define a restriction
\[\chi_1(w)\coloneqq\chi(w,(1,2,\dots,k))\]
for $w\in \N^n$. For instance, if $\chi=\chi^{\inv}$, then $\chi^{\inv}_1(w)=1$ if $w$ has a descent in position 1 and is 0 otherwise.

\begin{theorem}\label{thm:k-transitive-sn-invariant-inv-and-sum}
    Let $k\leq n$ be positive integers.
    Let $\chi:\N^n\times I_{n,k}\to \R$ be a $k$-transitive function and let $W\subseteq \N^n$ be an $\S_n$-invariant set of words. With $f_{\chi}$ and $g_{\chi}$ as in \Cref{eq:f_and_g_sub_chi}, we have
\begin{enumerate}[(a)]
    \item $\E_W[f_\chi]=\binom{n}{k}\E_W[\chi_1]$,
    \item $\E_W[g_\chi]=(n-k+1)\E_W[\chi_1]$, and
    \item $\E_W[f_\chi]=\frac{1}{k}\binom{n}{k-1}\E_W[g_\chi]$.
\end{enumerate}
\end{theorem}

\begin{proof}
    For a fixed element $\textbf{p}=(i_1,i_2,\dots,i_k)\in I_{n,k}$, fix $\sigma$ to be a permutation such that $\sigma(i_s)=s$ for each $s\in [k]$. Since $W$ is $\S_n$-invariant, for all $w\in W$ we have
    \[\sigma\cdot w=(w_{\sigma^{-1}(1)},w_{\sigma^{-1}(2)},\dots,w_{\sigma^{-1}(n)})\in W.\]
     Since $\chi$ is $k$-transitive, we can use  \Cref{eq:k-transitive-invariance-of-stat} with $\textbf{q}=(1,2,\dots,k)$ to get that for any $y\in \R$
    \[\P_W(\chi(w,\textbf{p})=y)=\P_W(\chi_1(\sigma \cdot w)=y)=\P_W(\chi_1(u)=y),\]
    where $\sigma$ is the permutation such that $\sigma(i_s)=s$ for all $s\in [k]$.
    Therefore, by linearity of expectation, we have 
    \begin{align}\label{eq:E(inv)1}
        \mathbb{E}_{W}\lb f_\chi\rb &=\sum_{\textbf{p}\in I_{n,k}}\mathbb{E}_{W}\lb \chi(w,\textbf{p})\rb=\sum_{\textbf{p}\in I_{n,k}}\E_W[\chi_1]=\binom{n}{k}\E_W[\chi_1],
    \end{align}
    since $|I_{n,k}|=\binom{n}{k}$, proving (a).

   For part (b), we proved in part (a) that $\P_W\lp \chi(w,(i,i+1,\dots,i+k-1)=y\rp=\P_W\lp \chi_1(w)=y\rp$ for all $i\in [n-k]$. 
   Hence, 
    \begin{equation*}\label{eq:E(des)1}
    \E_W\lb g_\chi\rb=\sum_{i=1}^{n-k+1}\E_W\lb 
    \chi(w, (i,i+1,\dots,i+k-1))\rb=(n-k+1)\E_W[\chi_1],
    \end{equation*}
    as desired.

    Finally, part (c) follows from combining (a) and (b).
\end{proof}

By selecting $\chi$ to be the 2-transitive function $\chi^{\inv}$ and by the definition of expectation applied to indicator random variables, we immediately obtain the following. 

\begin{corollary}\label{thm:inv-and-des-sums-for-words}
    Let $W\subset \N^{n}$ be an $\S_n$-invariant set of words of positive integers. 
    Then,
    \begin{enumerate}[label=(\alph*)]
        \item $\E_{W}\left[\inv\right]=\binom{n}{2}\mathbb{P}_{W}\left(\chi^{\inv}_1=1\right)$,
        \item $\E_{W}\left[\des\right]=(n-1)\P_{W}\left(\chi^{\inv}_1=1\right)$, and
        \item $\E_{W}\left[\inv\right]=\frac{n}{2}\E_{W}\left[\des\right]$.
    \end{enumerate}
\end{corollary}
Note that part (c) of \Cref{thm:inv-and-des-sums-for-words} is equivalent to  \begin{equation}\sum_{w\in W}\inv(w)=\frac{n}{2}\sum_{w\in W}\des(w).\end{equation}
From this, we can therefore derive the desired analogue of Schumacher's result \eqref{eq: Schumacker descents} for inversions since $\PF_n$ is $\S_n$-invariant.
\begin{corollary}\label{thm:total-number-of-inversions-in-parking-functions}
    For $n\geq 1$, the total number of inversions across all parking functions is
    \[\sum_{\alpha\in \PF_n}\inv(\alpha)=\frac{n(n+1)^{n-2}}{2}\binom{n}{2}.\]
\end{corollary}
Let $\Fub_n=|\C_n|$ be the $n$-th Fubini number defined by
\[\Fub_n=\sum _{{k=0}}^{n}\sum _{{j=0}}^{k}(-1)^{{k-j}}{\binom  {k}{j}}j^{n}.\]
The sequence of Fubini numbers is OEIS sequence \seqnum{A000670}. We obtain similar results for unit interval parking functions.
\begin{corollary}\label{thm:total-num-descent-in-cayley-permutations}
    Let $n\geq 1$.
    \begin{enumerate}
    \item The total number of descents across all unit interval parking functions, (or equivalently all Cayley permutations) is 
    \[\sum_{\alpha\in \UPF_n}\des(\alpha)=\frac{n-1}{2}(\Fub_n-\Fub_{n-1}). \]
   \item The total number of inversions across all unit interval parking functions, (or equivalently all Cayley permutations) is
   \[\sum_{\alpha\in \UPF_n}\inv(\alpha)=\frac{n(n-1)}{4}(\Fub_n-\Fub_{n-1}). \]
   \end{enumerate}
\end{corollary}

\begin{proof}
By \Cref{thm:bijection upf and cayley}, we can work with Cayley permutations. 

First observe that \[\P_{\C_n}(\chi_1^{\inv}=1)=\P_{\C_n}(\alpha_1>\alpha_2)=\P_{\C_n}(\alpha_1<\alpha_2)=\frac{1}{2}(1-\P_{\C_n}(\alpha_1=\alpha_2)).\]
The map \[\{\alpha\in \C_n\mid\alpha_1=\alpha_2\}\rightarrow \C_{n-1},\]
defined by $\alpha=(\alpha_1,\alpha_2,\alpha_3,\dots,\alpha_n)\mapsto(\alpha_2,\alpha_3,\dots,\alpha_n)$ is a bijection, so we have
\begin{equation}\label{eq:prop-of-tie-in-pos-1-in-upf}
\P_{\C_n}(\alpha_1=\alpha_2)=\frac{|\C_{n-1}|}{|\C_n|}=\frac{\Fub_{n-1}}{\Fub_n},
\end{equation}
and thus
\begin{equation}\label{eq:prop-of-des-in-pos-1-in-upf}
\E_{\C_n}(\chi_1^{\inv})=\P_{\C_n}(\chi_1^{\inv}=1)=\frac{\Fub_n-\Fub_{n-1}}{2\Fub_n}.\end{equation}
Applying  \Cref{thm:inv-and-des-sums-for-words} part (b) to \eqref{eq:prop-of-des-in-pos-1-in-upf} we get 
\begin{align}
    \sum_{\alpha\in \UPF_n}\des(\alpha) &=\Fub_n \E_{\C_n}[\des]\notag\\
    &=\Fub_n (n-1) \P_{\C_n}(\chi_1^{\inv}=1)\notag\\
    &=\frac{n-1}{2}(\Fub_n-\Fub_{n-1}).\label{eq:needed}
\end{align}
Applying \Cref{thm:inv-and-des-sums-for-words} part (c) to \eqref{eq:needed} yields
 \[\sum_{\alpha\in \UPF_n}\inv(\alpha)=\frac{n(n-1)}{4}(\Fub_n-\Fub_{n-1}). \qedhere\]
\end{proof}

\begin{remark}
    \Cref{thm:k-transitive-sn-invariant-inv-and-sum} is a statement about any finite $\S_n$-invariant set of words $W\subset \N^n$, so perhaps  it is of interest to view the result as a statement about the $\S_n$-module $\CC[W]$ or about its Frobenius image. 
    However, it is unclear at this time what the representation theoretic significance is.
\end{remark}

   The only property of $\S_n$ used in \Cref{thm:k-transitive-sn-invariant-inv-and-sum} is that it acts $k$-transitively on $[n]$.  Thus, we have the following generalization.
    \begin{proposition}\label{prop:2-transitive-G-invariant-words}
      Let $k\leq n$ be positive integers. Let $G$ be a group that acts $k$-transitively on $[n]$ and let $W\subset \N^{n}$ be a $G$-invariant set of words of positive integers. For any $k$-transitive function $\chi:\N^n\times I_{n,k}\to \R$, we have
\begin{enumerate}[(a)]
    \item $\E_W[f_\chi]=\binom{n}{k}\E_W(\chi_1)$,
    \item $\E_W[g_\chi]=(n-k+1)\E_W(\chi_1)$, and
    \item $\E_W[f_\chi]=\frac{1}{k}\binom{n}{k-1}\E_W(g_\chi)$.
\end{enumerate}
    \end{proposition}
    \begin{example}\label{ex:alternating_subgroup}    
    Let $\S_n^+$ denote the alternating subgroup of $\S_n$. Then using the work of D{\'e}sarm{\'e}nien and Foata \cite{desarmenien1992signed}, Fulman et al. in \cite[Eqn.~(7)]{fulman2021jointdistributiondescentssigns}) note that if $A_n(q)=\sum_{w\in \S_n}q^{\des(w)}$, then 
    \[\sum_{\pi\in \S_{2n}^{+}}q^{\des(\pi)}=2A_{2n}(q)+(1-q)^nA_{n}(q)\qquad \mbox{and} \qquad \sum_{\pi\in \S_{2n+1}^{+}}q^{\des(\pi)}=2A_{2n+1}(q)+(1-q)^nA_{n+1}(q).\]
    Considering both equations separately, we take a derivative with respect to $q$ of both sides and set $q=1$, to arrive at
\begin{equation}\label{eq:des_work}
\sum_{\pi\in\S_{2n}^{+}}\des(\pi)=2\sum_{\pi\in\mathfrak{S}_{2n}}\des(\pi) \qquad \mbox{and} \qquad \sum_{\pi\in\S_{2n+1}^{+}}\des(\pi)=2\sum_{\pi\in\mathfrak{S}_{2n+1}}\des(\pi).\end{equation}
We give more details on the history and counts of descents and inversions in permutations at the start of the next section, and use those counts along with the equations in \eqref{eq:des_work} to determine that

    \begin{equation}
        \sum_{\pi\in \S^+_{n}}\des(\pi)=\frac{1}{2}\sum_{\pi\in \S_{n}}\des(\pi)=\frac{n!(n-1)}{4}.
    \end{equation}
    Since the alternating group $\S_n^{+}$ acts 2-transitively on $[n]$ for $n\geq 4$, we can conclude by \Cref{prop:2-transitive-G-invariant-words} that 
    \begin{equation}
         \sum_{\pi\in \S^+_{n}}\inv(\pi) = \frac{n!n(n-1)}{8}.
    \end{equation}
    As a consequence, one can also see that total number of descents or inversions is the same in $\S_n\setminus \S_n^{+}$ as there is in $\S_n^+$. The equality of the total number of descents can be shown combinatorially via the involution of Wachs \cite{WACHS199259}.
\end{example}

\subsection{Applications}\label{sec:applications-of-linearity-of-expectation}
The method of enumerating statistics by way of computing the expectation of a carefully chosen $k$-transitive function has numerous applications. In particular, one can apply 
\Cref{thm:k-transitive-sn-invariant-inv-and-sum} straightforwardly in the following scenarios to obtain results, both old and new.

\subsubsection{Inversions and descents in permutations}
Counting the total number of inversions \cite[\seqnum{A001809}]{OEIS} is a very old problem from 1838 posed by Stern \cite{Stern_1838} and solved by Terquem \cite{Terquem1838}. 
The sequence enumerating the total number of descents is known as the \textit{Lah numbers} \cite[\seqnum{A001286}]{OEIS}. Note we used these results in \Cref{ex:alternating_subgroup}, now we prove them independently by using \Cref{thm:inv-and-des-sums-for-words}.
\begin{corollary}
    \label{thm:inversions-and-decents-in-permutations}The number of descents and inversions in $\S_n$ is given by 
    \[\sum_{w\in \S_n}\des(w)=\frac{n!(n-1)}{2}\quad \text{and}\quad \sum_{w\in \S_n}\inv(w)=\frac{n}{2}\frac{n!(n-1)}{2}=\frac{n!n(n-1)}{4},
    \]
    respectively.
\end{corollary}
\begin{proof}
Recall that $\chi^{\inv}:\N^n\times I_{n,2}\to \{0,1\}$ is defined by 
\[\chi^{\inv}(w,(i,j))=\begin{cases}
    1&w_i>w_j\\
    0&w_i\leq w_j \eqcom\\
\end{cases}\]
and that this has $f_{\chi^{\inv}}=\inv$ and $g_{\chi^{\inv}}=\des$. 
For $w\in \S_n$, the probability that $\chi^{\inv}(w,(1,2))=1$ is exactly~$\frac12$. Hence, by \Cref{thm:inv-and-des-sums-for-words} part (b), we have that 
\[\E_{\S_n}[\des]=(n-1)\frac{1}{2} \eqpd \]
Thus, we have
\[\sum_{w\in \S_n}\des(w)=\frac{n!(n-1)}{2}.\]
In an analogous fashion, 
applying  \Cref{thm:inv-and-des-sums-for-words} part (c) yields
\[\sum_{w\in \S_n}\inv(w)=\frac{n}{2}\frac{n!(n-1)}{2}=\frac{n!n(n-1)}{4}.\]
\end{proof}

\subsubsection{Ties in words} A \textbf{tie} in a word $w\in \N^n$ is a position $i$ such that $w_i=w_{i+1}$. The number of ties of $w$ is denoted $\tie(w)$. We have the following application of the aforementioned technique. 

\begin{corollary}
    \label{thm:ties-for-general-words}
    Let $w$ be a word with content $(c_1 , c_2, \dots, c_\ell)$. Then 
    \[\sum_{u\in \S_n\cdot w}\tie(u)=\frac{2}{n}\sum_{i=1}^\ell \binom{n}{c_1,c_2,\dots,c_{i-1},c_i-2,2,c_{i+1},\dots,c_\ell}\eqpd \]
\end{corollary}
\begin{proof}
Define \[\chi^{\tie}(w,(i,j))=\begin{cases}
    1&w_i=w_j\\
    0&w_i\neq w_j.
\end{cases}\]
Then, $\tie(w)=g_{\chi^{\tie}}(w)$. If $\con(w)=(c_1,c_2,\dots,c_\ell)$, then \[f_{\chi^{\tie}}(w)=\binom{c_1}{2}+\binom{c_2}{2}+\cdots+ \binom{c_\ell}{2}\] as this function simply selects pairs of numbers which are the same. If $W=\S_n \cdot w$, then $|W|=\binom{n}{c_1,c_2,\dots,c_\ell}$ and $f_{\chi^{\tie}}(u)=f_{\chi^{\tie}}(w)$ for all $u\in W$, so \[\sum_{u\in\S_n \cdot w }f_{\chi^{\tie}}(u)=\binom{n}{c_1,c_2,\dots,c_\ell}\left(\binom{c_1}{2}+\cdots +\binom{c_\ell}{2}\right)=\sum_{i=1}^\ell \binom{n}{c_1,c_2,\dots,c_{i-1},c_i-2,2,c_{i+1},\dots,c_\ell}.\]
Thus, the total number of ties is 
\[\sum_{u\in \S_n\cdot w}\tie(u)=\frac{2}{n}\sum_{u\in \S_n\cdot w}f_{\chi^{\tie}}(u)=\frac{2}{n}\sum_{i=1}^\ell \binom{n}{c_1,c_2,\dots,c_{i-1},c_i-2,2,c_{i+1},\dots,c_\ell}.\qedhere\]
\end{proof}
For some particular choices of $W$, \Cref{thm:ties-for-general-words} has interesting interpretations. In particular, we are able to derive nicer formulas for the total number of ties across all Cayley permutations (\Cref{prop:ties-in-cayley}) and across all parking functions (\Cref{cor:ties in parking functions}), leading us to wonder if there exist bijective proofs that the formulas we can derive from \Cref{thm:ties-for-general-words} are equal to these simpler expressions.

\begin{proposition}\label{prop:ties-in-cayley}
    The total number of ties across all Cayley permutations of length $n$ is
    \[\sum_{w\in \C_n}\tie(w)=(n-1)\Fub_{n-1}.\]
\end{proposition}
\begin{proof}
      This follows from \Cref{eq:prop-of-tie-in-pos-1-in-upf}, which says that
    \[\E_{\C_n}[\chi_1^{\tie}]=\frac{|\C_{n-1}|}{|\C_n|}=\frac{\Fub_{n-1}}{\Fub_n},\]
    where $\Fub_n$ denotes the $n$-th Fubini number. 
\end{proof}

    Combining \Cref{prop:ties-in-cayley} with \Cref{thm:ties-for-general-words},  yields the identity 
    \[\frac{2}{n}\sum_{c\in\Comp(n)}\sum_{i=1}^\ell \binom{n}{c_1,c_2,\dots,c_{i-1},c_i-2,2,c_{i+1},\dots,c_\ell}=(n-1)\Fub_{n-1}\eqpd \]
    Finding a bijective proof of this fact remains an open problem; see \Cref{sec:future-work} for a formal statement.

    \begin{remark}        
    If $\Fub(z)=\sum_{n\geq 0}\Fub_{n}z^{n}/n!$, then recall that $\Fub(z)=\frac{1}{2-e^z}$ and thus
    \[\Fub(z)'=\sum_{n\geq 1}\Fub_{n}\frac{z^{n-1}}{(n-1)!}=\sum_{n\geq 2}(n-1)\Fub_{n-1}\frac{z^{n-2}}{(n-1)!}=\sum_{n\geq 1} \lp \sum_{\alpha\in\C_n}\tie(\alpha)\rp \frac{z^{n-2}}{(n-1)!}\]
    by \Cref{prop:ties-in-cayley}. Since $\Fub(z)=(2-e^z)^{-1}$, we have
    \[\sum_{n\geq 1} \lp \sum_{\alpha\in\C_n}\tie(\alpha)\rp \frac{z^{n-2}}{(n-1)!}=\lp\frac
    {1}{2-e^z}\rp=\frac{e^z}{(2-e^z)^2}.\]
    Hence,
    \[\sum_{n\geq 0}\lp \sum_{w\in \C_{n+1}}\tie(w)\rp\frac{z^n}{n!}=\frac{ze^z}{(2-e^z)^2}.\]
    \end{remark} 
      
   For parking functions, as with Cayley permutations, we can get a nice formula for the total number of ties without invoking \Cref{thm:ties-for-general-words}. 
    \begin{lemma}\label{prop:parktie}
    The number of parking functions of length $n$ having a tie in position 1 is given by \[(n+1)^{n-2} \eqpd \]
    \end{lemma}
   
        \begin{proof}
We begin by recalling the map of Pollak \cite{RIORDAN1969408}: Given a sequence $(w_1,w_2,\dots,w_n)\in [n+1]^n$, we can imagine cars parking on a circular street with $n+1$ spots with preference list $w$.  Since there are $n$ cars and $n+1$ spots, there is empty spot. If $i$ is the empty spot, set
    \[\alpha=(w_{1}-i,w_{2}-i,\dots,w_{n}-i),\]
    where all entries are taken modulo $n+1$. This defines a parking function $\alpha\in \PF_n$ and hence an $(n+1)$-to-1 map 
    $\psi:[n+1]^n\to \PF_n$ as proven in \cite{foata1974mappings}. Then we can see that $\psi(w)$ is a parking function with a tie in position 1 if and only if $w$ has a tie in position $1$. Since there are $(n+1)^{n-1}$ sequences $w\in [n+1]^n$ such that $w$ has a tie in position 1, and since $\psi$ is $(n+1)$-to-1, we have
    \[|\{\alpha\in \PF_n\mid \alpha_1=\alpha_2\}|=\frac{(n+1)^{n-1}}{n+1}=(n+1)^{n-2}.\qedhere\]
    \end{proof}
    As a consequence of \Cref{thm:k-transitive-sn-invariant-inv-and-sum}, we have the following.
    \begin{corollary}\label{cor:ties in parking functions}For $n\geq 1$,
            \[\sum_{\alpha\in \PF_n}\tie(\alpha)=(n-1)(n+1)^{n-2}.\]
    \end{corollary}
    
    \begin{remark}\Cref{cor:ties in parking functions} can also be derived from \cite[Theorem 13]{bib:DescentsInPF} which says that if  $\PF_{(n,i)}$ is the set of parking functions with $i$ ties, then
    \[|\PF_{(n,i)}|=\binom{n-1}{i}n^{n-1-i}.\]
    Hence, we have
    \begin{align*}
        \sum_{\alpha\in \PF_n}\tie(\alpha)
        =\sum_{i=0}^{n-1}i\binom{n-1}{i}n^{n-1-i}
        =(n-1)\sum_{i=1}^{n-1}\binom{n-2}{i-1}n^{n-1-i}
        =(n-1)(n+1)^{n-2}, 
    \end{align*}
    where the final equality is an application of the binomial theorem.\end{remark}

    By Schumacher's result on the number of ties, we can write \Cref{cor:ties in parking functions} as
    \begin{equation}\label{eq:des_ties}\sum_{\alpha\in \PF_n}\des(\alpha)=\frac{n}{2}\sum_{\alpha\in \PF_n}\tie(\alpha).\end{equation}
    The equality in \eqref{eq:des_ties} does not hold in general for all $\S_n$-invariant sets of words. 
    For example, the  equality $\sum_{w\in W} \des(w)=\frac{n}{2}\sum_{w\in W}\tie(w)$ does not hold when $W=\S_n$.
    Thus, we can apply \Cref{thm:k-transitive-sn-invariant-inv-and-sum} to conclude that
    \[\sum_{\alpha\in \PF_n}\des(\alpha)=\sum_{\alpha\in \PF_n}f_{\chi^{\tie}}(\alpha).\]
    This indicates that there is a bijection
    \[\{(\alpha,i)\mid \alpha\in \PF_n, i\in \Des(\alpha)\}\to \{(\alpha,(r,s))\mid \alpha\in \PF_n, 1\leq r<s\leq n, \alpha_r=\alpha_s\}.\]
    We leave it as an open problem to find this bijection explicitly in \Cref{sec:future-work}.
    
 To get a formula for the number of ties across all parking functions using \Cref{thm:ties-for-general-words}, we recall that  parking functions can be described as the collection of words in $[n]^n$ whose content $\con(w)=(c_1,c_2,\dots,c_n)$ satisfies
    \begin{equation}\label{eq:Hess}
    i\leq c_1+c_2+\cdots+c_i\leq n,
    \end{equation}
    for all $i\in [n]$.
    Let $\mathrm{Hess}(n)$ denote the set of sequences $(c_1,c_2,\dots,c_n)$ of nonnegative integers satisfying \Cref{eq:Hess}.\footnote{The notation comes from literature of ``Hessenberg functions'' which are functions $h:[n]\to [n]$ satisfying $i\leq h(i)$ for all $i\in [n]$ and $h(i)\leq h(i+1)$ for all $i\in [n-1]$.
    These functions are a central object in the work on chromatic quasisymmetric functions of Shareshian and Wachs \cite{shareshian_chromatic_2016}.} 
    Then we may apply \Cref{thm:ties-for-general-words} and \ref{cor:ties in parking functions} to establish the following.
    \begin{corollary}\label{cor:mystery of parking function ties}
    For $n\geq 1$,
         \[\sum_{c\in \mathrm{Hess}(n)}\sum_{i=1}^\ell \binom{n}{c_1,c_2,\dots,c_{i-1},c_i-2,2,c_{i+1},\dots,c_\ell}=\binom{n}{2}(n+1)^{n-2}.\]
    \end{corollary}
    We again leave it as an open problem to find this bijection explicitly, and state this formally in \Cref{sec:future-work}.
\subsubsection{Descent tops and inversion tops}
Recall that the \textbf{major index} of a permutation $w$ is  $\maj(w)=\sum_{i\in \Des(w)}i$, or the sum of the indices where a decent occurs in $w$. This can be thought of as summing over descents weighted by the position they occur in. Instead of weighting by position, we can weight by its value and arrive at the \textbf{descent top statistic} defined by
\[\dtop(w)=\sum_{i\in \Des(w)}w_i.\]
For sequence data on the descent top statistic see \seqnum{A263753}. For a word $w\in \N^n$, define 
     \[\chi^{\dtop}(w,(i,j))=\begin{cases}
         w_i&w_i>w_{j}\\
         0&\text{else,}
     \end{cases}\]
     The function $\chi^{\dtop}$ is a $2$-transitive function and $\dtop=g_{\chi^{\dtop}}$. 
     We can also define the \textbf{inversion top statistic} as
     \[f_{\chi^{\dtop}}(w)=\itop(w)=\sum_{i<j}\chi^{\dtop}(w,(i,j)).\]
     For $r\in [n]$, the probability that $\chi^{\dtop}((w,(1,2))=r$ is $\frac{1}{n}\frac{r-1}{n-1}$ because you need to select a permutation with $w_1=r$ and $w_{2}<r$.      Hence, we have
     \begin{align*}   
      \E_{\S_n}[\chi_{1}^{\dtop}]&=\sum_{r=1}^n \frac{r-1}{n(n-1)}\cdot r\\
     &= \frac{1}{n(n-1)}\sum_{r=1}^n r^2-r\\&=\frac{1}{n(n-1)}\lp \frac{n(n+1)(2n+1)}{6}-\frac{n(n+1)}{2}\rp\\&=\frac{n+1}{3}.
     \end{align*}
     Thus, we have
     \begin{equation}\label{eq:ttl-num-of-wdes-winvs}\sum_{w\in \S_n}\dtop(w)=\frac{(n+1)!(n-1)}{3}\qquad \sum_{w\in \S_n}\itop(w)=\frac{(n+1)!n(n-1)}{6}=\binom{n+1}{3}n!.\end{equation}

\subsubsection{Graphical inversions}
Let $D$ be a directed graph on vertex set $[n]$ and directed edge set $E$, the elements of which we represent using arrows $i\to j\in E$. For a word $w\in [n]^n$, define 
\[\chi^{D}(w,(i,j))=\begin{cases}
    1&w_i\to w_j\in E\\
    0&w_i\to w_j\not\in E.
\end{cases}\]
If we define  $\theta(x,y)=1$ if $x\to y\in E$ and $0$ otherwise for $x,y\in [n]$, then $\chi^D(w,(i,j))=\theta(w_i,w_j)$ for all $w\in [n]^n$, so this is a $2$-transitive function. The function $f_{\chi^{D}}$ is the \textbf{graphical inversion statistic}, denoted  $\inv_D'$, studied by Foata and Zeilberger in \cite{FOATA199679}, where it was studied alongside the analogously defined \textit{graphical major index} We have
\begin{equation}\label{eq:Gmaj-chi1}\E_{\S_n}[\chi_1^D]=\frac{|E|}{n(n-1)}\end{equation}
since $\chi^D_1(w)=1$ if and only if $w_1\to w_2\in E$ for each $w\in \S_n$. Defining the \textbf{graphical descent statistic } as $\des_D'=g_{\chi^D}$, we have
\begin{equation}\label{eq:graphical-equations}
\sum_{w\in \S_n}\inv_D'(w)=n!\binom{n}{2}\frac{|E|}{n(n-1)}=\frac{n!|E|}{2}\quad \mbox{and}\quad\sum_{w\in \S_n}\des_D'(w)=n!(n-1)\frac{|E|}{n(n-1)}=(n-1)!|E|.
\end{equation}

by combining \Cref{eq:Gmaj-chi1} with \Cref{thm:k-transitive-sn-invariant-inv-and-sum}.

If we take $D$ to be the complete graph with edges of the form $i\to j$ whenever $i>j$, we have $\chi^D=\chi^{\inv}$ and $|E|=\binom{n}{2}$ is the number of edges, so \Cref{eq:graphical-equations} yield the formulas in \Cref{thm:inversions-and-decents-in-permutations}.

\subsubsection{Descents, big and small}
Instead of looking at all inversions, we can restrict our attention to inversions of a particular form. For a word $w\in \N^n$, indices $i$ and $j$  define
\[\chi^{\sdes}(w,(i,j)) =   \begin{cases}
    1 & w(i)=w(j)+1\\
    0 & \text{else.}
\end{cases}  \]
 
Then we define $\sdes= g_{\chi^{\sdes}}$
which counts the number of \textbf{small descents} of $w$, i.e., positions $i$ where $w(i)=w(i+1)+1$, and likewise $\sinv=f_{\chi^{\sdes}}$ counts \textbf{small inversions}, i.e., positions $i<j$ where $w(i)=w(j)+1$. Observe that if $\sigma$ is a permutation, then 
\begin{equation}\label{eq:sinv-equals-ides}\sinv(\sigma)=\des(\sigma^{-1})\end{equation}
and so 
\begin{equation}\label{eq:sum-sinv-equals-ides}\sum_{\sigma\in \S_n}\sinv(\sigma)=\sum_{\sigma\in \S_n}\des(\sigma)=\frac{n!(n-1)}{2}.\end{equation}
One advantage of using $\sinv(\sigma)$ over $\des(\sigma^{-1})$ is that the former is more easily generalizable to words. 

We can also define \textbf{big descents} and \textbf{big inversions} as $\bdes=g_{\chi^{\bdes}}$ and $\binv=f_{\chi^{\bdes}}$, respectively, where
\[\chi^{\bdes}(w,(i,j)) =   \begin{cases}
    1 & w(i)>w(j)+1\\
    0 & \text{else.}
\end{cases}  \]
for a given word $w\in \N^n$
and pair $1\leq i<j\leq n$. Note that $\chi^{\inv}=\chi^{\sdes}+\chi^{\bdes}$. 
Big descents correspond to $1$-descents of Foata and Schutzenberger \cite{FoaSch70}. These were more recently studied by Shareshian and Wachs in a quasisymmetric function generalization of Eulerian polynomials \cite{shareshian_eulerian_2010,shareshian_chromatic_2016} and by Elizalde, Rivera, Zhuang on sets of pattern avoiding permutations \cite{elizalde2024countingpatternavoidingpermutationsbig}. 
To compute expectations, note that a permutation starting with a small descent starts with a number $i$ bigger than $1$ and then $i-1$. Thus, it follows that 
\begin{equation}\label{eq:exp-of-1-sdes}\E_{\S_n}[\chi_1^{\sdes}]=\frac{(n-1)\cdot 1}{n(n-1)}=\frac{1}{n}=1-\E_{\S_n}[\chi_1^{\bdes}].\end{equation}
  Thus, combining \Cref{eq:exp-of-1-sdes} and \Cref{thm:k-transitive-sn-invariant-inv-and-sum} (or equivalently \Cref{eq:sum-sinv-equals-ides}, \Cref{thm:k-transitive-sn-invariant-inv-and-sum}, and \Cref{thm:inversions-and-decents-in-permutations}), we have the following.
  \begin{proposition} \label{prop:sdes_sinv_bdes_binvSN}
      Suppose $n\geq 1$.
      \begin{enumerate}
          \item $\sum_{\sigma\in \S_n}\sdes(w)=(n-1)(n-1)!$.
          \item $\sum_{\sigma\in \S_n}\sinv(w)=\binom{n}{2}(n-1)!=\frac{n!(n-1)}{2}$.
          \item $\sum_{\sigma\in \S_n}\bdes(w)=\frac{(n-1)(n-2)(n-1)!}{2}$.
          \item $\sum_{\sigma\in \S_n}\binv(w)=\frac{(n-1)(n-2) n!}{4}$.

      \end{enumerate}
  \end{proposition}
Items (1) and (3) appear to be well-known; see \seqnum{A001563} and \seqnum{A001804}, respectively.

For parking functions, the statistic $\sinv$ counts the number of pairs of cars with adjacent preferences $i$ and $i+1$ so that the second car has the smaller preference. Similarly, $\sdes$ counts the number of pairs of adjacent cars, car $i$ and car $i+1$, with adjacent preferences $j$ and $j-1$, respectively. 

The Pollak map $[n+1]^n\to \PF_n$ (discussed in \Cref{prop:parktie}) sends $w=(a_1,\dots, a_n)$ to a parking function $\alpha$ of the form $(a_1-i,\dots, a_n-i)$ for some $i$, where the subtraction is computed modulo $n+1$.  If $\chi_1^{\sdes}(\alpha)=1$, then $a_1-i=a_2-i+1\pmod{n+1}$, so we must have that $a_1-a_2=1 \pmod{n+1}$  On the other hand, If $w_1=w_{2}+1$ or $w_1=1$ and $w_2=n$, then $i\not\in \{w_1,w_2\}$ since both $w_1$ and $w_2$ must be occupied, so $\chi_1^{\sdes}(\alpha)=1$. Hence, $\alpha$ has $\chi_1^{\sdes}(\alpha)=1$ if and only if $a_1-a_2=1 \pmod{n+1}$. For $w=(a_1,a_2,\dots,a_n)\in W=[n+1]^n,$  the probability $a_1-a_2=1\pmod{n+1}$ is $\frac{1}{n+1}$ since the first letter can be freely chosen and the second letter is uniquely determined by the first, so
\begin{equation}\label{eq:exp-of-1-sdes-pf}\E_{\PF_n}[\chi_1^{\sdes}]=\frac{1}{n+1}.\end{equation} 
  Thus, combining \Cref{eq:exp-of-1-sdes-pf} and \Cref{thm:k-transitive-sn-invariant-inv-and-sum}, we have the following.
\begin{proposition}\label{prop:sdes_sinv_bdes_binvPF}
    Let $n\geq 1$.
\begin{enumerate}
    \item $\sum_{w\in \PF_n} \sdes(w) = (n-1)(n+1)^{n-2}$ .
    \item $\sum_{w\in \PF_n} \sinv(w) = \binom{n}{2}(n+1)^{n-2}$.
    \item $\sum_{w\in \PF_n}\bdes(w)=\binom{n-1}{2}(n+1)^{n-2}$.
    \item $\sum_{w\in \PF_n}\binv(w) = \frac{n(n-1)(n-2)(n+1)^{n-2}}{4}$.
\end{enumerate}
\end{proposition}

\subsubsection{Permutation patterns} We finish our probabilistic discussion with a closer look 
at when $W=\S_n$. The archetypal example of $k$-transitive function  on $\S_n$ is a test for the existence of a ``pattern'' of length $k$ in a permutation. 
For instance, a $132$-pattern of a permutation $w$ is any triple of increasing indices $(i_1,i_2,i_3)$ such that $w_{i_1}<w_{i_3}<w_{i_2}$ and $w_{i_1}<w_{i_2}$. 
In this light, an inversion is a $21$-pattern. 
More generally, for fixed $k$, we can define a \textbf{pattern} with a permutation $\rho\in \S_k$, namely, a function $\chi^\rho:\N^n\times I_{n,k}$ by setting
     \[\chi^\rho(w,(i_1,i_2,\dots,i_k))=\begin{cases}
         1&\text{if }w_{i_{\rho^{-1}(1)}}< w_{i_{\rho^{-1}(2)}}<\cdots< w_{i_{\rho^{-1}(k)}}\\
         0&\text{otherwise}.
     \end{cases}\]
    If $\chi^\rho(w,\textbf{p})=1$, we say that the pattern $\rho$ \textbf{occurs} in $w$ in the positions indicated by the entries in $\textbf{p}$. 
    Then $f_{\chi^\rho}(w)$ is the number of occurrences of the pattern $\rho$ in $w$ and $g_{\chi^\rho}(w)$ is the number of occurrences of $\rho$ in adjacent positions. 
    Hence, by \Cref{thm:k-transitive-sn-invariant-inv-and-sum}, we have
     \begin{equation}
         \sum_{w\in \S_n} f_{\chi^\rho}(w)=\frac{1}{k}\binom{n}{k-1}\sum_{w\in \S_n} g_{\chi^\rho}(w).\label{eq:pattern-inv-des}
     \end{equation}
 Here are two simple observations about any pattern $\rho\in \S_k$: First,
 \begin{align}
     \label{eq:k-equidistribution}\chi^{\rho}(w,\textbf{p})&=\chi^{id}(u,\textbf{p}),\end{align}
     where $u$ is the permutation obtained by rearranging letters of $w$ in the positions  $\textbf{p}$ in increasing order. Second,
     \begin{align}
     \sum_{w\in \S_n} f_{\chi^\rho}(w)&=n!\binom{n}{k}\frac{1}{k!}\eqpd\label{eq:total-num-patterns}
 \end{align}
    For this, we can see that $\P_{\S_n}(\chi^{\rho}_1=1)=\frac{1}{k!}$ since there is one way to arrange the first $k$ letters to form the pattern $\rho$. 
    Since there are $n!$ permutations, the claim follows from \Cref{thm:k-transitive-sn-invariant-inv-and-sum} part (a). 
    Combining \Cref{eq:pattern-inv-des} and \Cref{eq:total-num-patterns}, we have 
    \begin{equation}\label{eq:total-num-adj-patterns}
        \sum_{w\in \S_n}g_{\chi^\rho}(w)=\frac{(n-k+1)n!}{k!}=(n-k+1)!\binom{n}{k},
    \end{equation}
    which can also be directly seen from \Cref{thm:k-transitive-sn-invariant-inv-and-sum} part (b).

    Some classical permutation statistics can be written as sums of pattern enumerators. Let $R\subseteq \S_k$. 
    One can see that 
    \begin{equation}
    \chi^{R}\coloneqq\sum_{\rho\in R} \chi^{\rho},
    \end{equation}
    is $k$-transitive, since a permutation $w$ can only have a single pattern occur on a given $k$-tuple of positions. Hence, from \Cref{eq:k-equidistribution},  \Cref{eq:total-num-patterns}, and  \Cref{eq:total-num-adj-patterns}, we have that
   \begin{equation}
        \E[f_{\chi^R}]=|R|\cdot \E[f_{\chi^{id}}]=|R|\cdot n!\binom{n}{k}\frac{1}{k!}\quad\text{and}\quad  \E[g_{\chi^R}]=|R|\cdot\E[g_{\chi^{id}}]=|R|\cdot\frac{(n-k+1)n!}{k!}.
   \end{equation}
   \begin{example}
   The classical \textbf{peak} number of a permutation $w$ is defined as 
 \[\operatorname{pk}(w)=|\{i\in [n-2]\mid w_i<w_{i+1}>w_{i+2}\}|.\]
Peaks of permutations are a classic topic in combinatorics and have been studied by a great many authors. See for example \cite{peak1,peak2,peak3,peak4,peaks10,peaks11,peaks5,peaks6,peaks7,peaks8,peaks9}.
 We define the \textbf{horizon} number of a permutation $w$ to be
 \[\operatorname{hz}(w)=\#\{1\leq i_1<i_2<i_3\leq n\mid w_{i_1}<w_{i_2}>w_{i_3}\}.\]
 Note the horizon number of a permutation is simply the number of $132$ and $231$ patterns. 
 Namely, if $R=\{132,231\}$, then  
 $\operatorname{hz}=f_{\chi^R}$ and $\operatorname{pk}=g_{\chi^R}$. Hence,
 \begin{equation}\label{eq:ttl-num-of-peaks}
 \sum_{w\in \S_n}\operatorname{pk}(w)=2\sum_{w\in \S_n}f_{\chi^{132}}=\frac{(n-2)n!}{3}
 \end{equation}
and 
 \begin{equation}
 \sum_{w\in \S_n}\operatorname{hz}(w)=\frac{1}{9}\binom{n}{2}(n-2)n!=\frac{1}{3}\binom{n}{3}n!.
 \end{equation}
\end{example}
\begin{remark}

Comparing \Cref{eq:ttl-num-of-wdes-winvs} and \Cref{eq:ttl-num-of-peaks} reveals that
    \begin{equation}
        \sum_{w\in \S_n}\dtop(w)=\sum_{w\in \S_{n+1}}\operatorname{pk}(w).
    \end{equation}
   We leave finding a combinatorial proof as an open problem. See \Cref{sec:future-work}.
\end{remark}

\subsection{Generating functions}
Given a family $\{W_n\}_{n\geq 1}$ of words so that $W_n\subseteq \Z_{>0}^n$ and $W_n$ is $\S_n$-invariant for all $n$, and given a $k$-transitive function $\chi$, define the exponential generating functions
\begin{align}\label{eq:prob_gen_fun}
F(z)=\sum_{n\geq 1} \lp \sum_{w\in W_n}f_\chi(w)\rp \frac{z^n}{n!}, \quad G(z)=\sum_{n\geq 1} \lp \sum_{w\in W_n}g_\chi(w)\rp \frac{z^n}{n!}, \quad \text{and} \quad H(z)=\sum_{n\geq 1} \lp \sum_{w\in W_n}\chi_1(w)\rp \frac{z^n}{n!}.
\end{align}
For example, if $\chi=\chi^{\inv}$ and $W_n=\S_n$, then $F(z)$ is the exponential generating function for the total number of inversions, $G(z)$ is that for descents, and $H(z)$ is that for descents in the first position.   
\begin{proposition}\label{prop:exp functions}
With $F(z)$, $G(z)$, and $H(z)$ as in \Cref{eq:prob_gen_fun},
we have the following relations:
\begin{align}
    F(z)&=\frac{z^{k-1}}{k!}\dfrac{d^{k-1}}{dz^{k-1}}[G(z)], \label{eq:F-to-G}\\
     G(z)&=zH'(z)-(k-1)H(z),\label{eq:G-to-H}\\
      \dfrac{d^{k-1}}{dz^{k-1}}\lb G(z)\rb &=z\dfrac{d^k}{dz^k}\lb H(z)\rb,\label{eq:Gdiff-to-Hdiff} \text{and}\\
      F(z) &=\frac{z^k}{k!}\dfrac{d^k}{dz^k}\lb H(z)\rb. \label{eq:F-to-H}
\end{align}
\end{proposition}
\begin{proof} These are routine computations following from \Cref{thm:k-transitive-sn-invariant-inv-and-sum}. To establish \Cref{eq:F-to-G}, we have 
\[F(z)=\sum_{n\geq 1}\frac{1}{k}\binom{n}{k-1}\lp\sum_{w\in W_n}g_{\chi}(w)\rp \frac{z^n}{n!}=\frac{z^{k-1}}{k!}\sum_{n\geq 1}\lp\sum_{w\in W_n}g_{\chi}(w)\rp \frac{z^{n-k+1}}{(n-k+1)!}=\frac{z^{k-1}}{k!}\dfrac{d^{k-1}}{dz^{k-1}}[G(z)].\]
Recall from \Cref{thm:k-transitive-sn-invariant-inv-and-sum}(b) that $\E_W[g_\chi]=(n-k+1)\E_W[\chi_1]=n\E_W[\chi_1]-(k-1)\E_W[\chi_1]$. Hence, we have 
\begin{align*}
G(z)=\sum_{n\geq 1}\lp \sum_{w\in W_n} g_\chi(w)\rp \frac{z^n}{n!}
=\sum_{n\geq 1}\lp \sum_{w\in W_n} \chi_1(w)\rp \frac{z^n}{(n-1)!}-(k-1)\sum_{n\geq 1}\lp \sum_{w\in W_n} \chi_1(w)\rp \frac{z^n}{n!},
\end{align*} 
and therefore \Cref{eq:G-to-H} follows. 

\Cref{eq:Gdiff-to-Hdiff} follows from using  the product rule applied to the \Cref{eq:G-to-H}. Lastly, \Cref{eq:F-to-H} follows from combining \Cref{eq:F-to-G} and \Cref{eq:Gdiff-to-Hdiff}. 
\end{proof}
\Cref{prop:exp functions} allows one to quickly get exponential generating functions for $\sum f_\chi$ and $\sum g_\chi$  given an expression for $H(z)$, which is usually the simplest to compute. 
Note that as far as exponential generating functions are concerned, it does not matter if one computes 
\[
G(z)=\sum_{n\geq 1} \lp \sum_{w\in W_n}g_\chi(w)\rp \frac{z^n}{n!}\quad \text{or }\quad \dfrac{d^{k-1}}{dz^{k-1}}[G(z)]=\sum_{n\geq 1} \lp \sum_{w\in W_n}g_\chi(w)\rp \frac{z^{n-k+1}}{(n-k+1)!}.
\] 
since both ``record'' the examined sequence in essentially the same way.

\begin{example}
    Consider $W_n=\UPF_n$ and $\chi=\chi^{\inv}$. 
    It is slightly more straightforward to compute $G'(z)=zH''(z)$. 
    Recall that
    \[\sum_{n\geq 0}\Fub_n \frac{z^n}{n!}=\frac{1}{2-e^z}.\]
    Combining this and \Cref{eq:prop-of-des-in-pos-1-in-upf} yields
    \begin{align*}
    2H'(z)
    &=2\sum_{n\geq 1} \lp \sum_{w\in \UPF_n}\chi^{\inv}_1(w)\rp \frac{z^{n-1}}{(n-1)!}\\
    &=\sum_{n\geq 1} (\Fub_n-\Fub_{n-1})\frac{z^{n-1}}{(n-1)!}\\
    &=\sum_{n\geq 1} \Fub_n\frac{z^{n-1}}{(n-1)!}-\sum_{n\geq 1} \Fub_{n-1}\frac{z^{n-1}}{(n-1)!}\\
    &=\lp \frac{1}{2-e^z}\rp'-\frac{1}{2-e^z}\\
    &=\frac{ 2( e^z-1)}{(2 - e^z)^2}.
    \end{align*}
    Thus, from \Cref{thm:total-num-descent-in-cayley-permutations},
    \begin{equation}
   G'(z)= \sum_{n\geq 1}\lp \sum_{\alpha\in \UPF_n} \des(w)\rp \frac{z^{n-1}}{(n-1)!}=z\lp \frac{ e^z-1}{(2 - e^z)^2}\rp'=\frac{ze^{2 z}}{(2-e^z)^3},
    \end{equation}
    and from \Cref{prop:exp functions},
    \begin{equation}
    F(z)=\sum_{n\geq 1}\lp \sum_{\alpha\in \UPF_n} \inv(w)\rp \frac{z^{n}}{n!}=\frac{z}{2!}G'(z)=\frac{z^2e^{2 z}}{2(2-e^z)^3}.
    \end{equation}
\end{example}

\begin{example}
    Consider $\chi=\chi^{\tie}$ on $W_n=\C_n$. From \Cref{eq:prop-of-tie-in-pos-1-in-upf} we have
    \[H'(z)=\sum_{n\geq 1} \lp\sum_{\alpha\in \C_n}\chi_1^{\tie}(\alpha)\rp\frac{z^{n-1}}{(n-1)!}= \sum_{n\geq 1}\Fub_{n-1}\frac{z^{n-1}}{(n-1)!}=\frac{1}{2-e^z}.\]
    Hence, 
    \begin{equation}G'(z)=\sum_{n\geq 1} \lp\sum_{\alpha\in \C_n}\tie(\alpha)\rp\frac{z^{n-1}}{(n-1)!}=\sum_{n\geq 1}\Fub_{n-1}\frac{z^{n-1}}{(n-2)!}=z\lp \frac{1}{2-e^z}\rp'=\frac{ze^z}{(2- e^z)^2}.\end{equation}
\end{example}

\section{Future work}\label{sec:future-work}
We  conclude with some directions for further study.
One can study inversions of parking functions for other variations and special subsets of parking functions, including
$k$-Naples parking functions, vacillating parking functions, 
and parking sequences and assortments. For a summary of some of these sets, refer to \cite{bib:Carlson+}. 
Thus, the program of this paper may be applied to them.
\begin{problem}
    For other families of parking functions, find nice expressions for their inversion generating functions. Furthermore, determine if there is a natural $\S_n$-action on the family and determine its Frobenius image.
\end{problem} 
Inversions are just one of many word statistics to consider on parking functions. Hence, one research agenda is as follows:
\begin{problem}
    Investigate generating functions for other word statistics on parking functions and their subsets and generalizations. 
\end{problem} 
Next, we have seen that some $q$-analogues can be explained via representation theoretic formulas. We would like fully representation theoretic explanations of all of the results in this paper. In particular, we pose the following.
\begin{problem}
    Determine the representation theoretic interpretation of \Cref{thm:k-transitive-sn-invariant-inv-and-sum} and of the equality $A_n^{\inv,\asc}(q,2)=\UPF_n(q)$ given in \Cref{ex:alternating_subgroup}.
\end{problem}
Finally, in \Cref{sec:applications-of-linearity-of-expectation}, our techniques have uncovered a few surprising formulas. Since our methods are purely probabilistic and/or algebraic, it would be quite interesting to find more concrete, bijective proofs of our results.
\begin{problem}\label{problem: combinatorial-proofs}
    Find combinatorial proofs of the following equations. For $n\geq 1$:
    \begin{enumerate}
        \item $ \displaystyle\sum_{w\in \S_n}\dtop(w)=\displaystyle\sum_{w\in \S_{n+1}}\operatorname{pk}(w)$.
        \item $\displaystyle\sum_{\alpha\in \PF_n}\des(\alpha)=\displaystyle\sum_{\alpha\in \PF_n}f_{\chi^{\tie}}(\alpha)$.
        \item $\frac{2}{n}\displaystyle\sum_{c\vDash n}\displaystyle\sum_{i=1}^\ell \binom{n}{c_1,c_2,\dots,c_{i-1},c_i-2,2,c_{i+1},\dots,c_\ell}=(n-1)\Fub_{n-1}$.
        \item $\displaystyle\sum_{c\in \mathrm{Hess}(n)}\displaystyle\sum_{i=1}^\ell \binom{n}{c_1,c_2,\dots,c_{i-1},c_i-2,2,c_{i+1},\dots,c_\ell}=\binom{n}{2}(n+1)^{n-2}$.
    \end{enumerate}
    \end{problem}

\section*{Acknowledgments}
The genesis for this project was the 2024 Graduate Research Workshop in Combinatorics, hosted by University of Wisconsin, Milwaukee and funded by the NSF (Grant No. DMS~--~1953445).
We thank the developers of
SageMath~\cite{sage} and CoCalc~\cite{SMC}, which were useful in this research.
We also thank Steve Butler and Mei Yin, for their comments and guidance at the start of this project, and Kimberly J. Harry, for initial contributions. Special thanks also to Ari Cruz, Jeremy L. Martin, Matt McClinton, and Keith Sullivan for helpful conversations at the initial stages of this collaboration. We thank Heesung Shin for a helpful correspondence regarding the bijection of Fran\c{c}on \cite{francon_acyclic_1975}, and Vasu Tewari for their helpful comments about \Cref{lem:character-of-words-inversions}. 

\bibliographystyle{plain}
\bibliography{bibliography.bib}

\begin{thebibliography}{10}

\bibitem{peak1}
Marcelo Aguiar, Kathryn Nyman, and Rosa Orellana.
\newblock New results on the peak algebra.
\newblock {\em J. Algebraic Combin.}, 23(2):149--188, 2006.

\bibitem{peak2}
Louis~J. Billera, Samuel~K. Hsiao, and Stephanie van Willigenburg.
\newblock Peak quasisymmetric functions and {E}ulerian enumeration.
\newblock {\em Adv. Math.}, 176(2):248--276, 2003.

\bibitem{peaks11}
Sara Billey, Krzysztof Burdzy, and Bruce~E. Sagan.
\newblock Permutations with given peak set.
\newblock {\em J. Integer Seq.}, 16(6):Article 13.6.1, 18, 2013.

\bibitem{peak3}
Pierre Bouchard, Hungyung Chang, Jun Ma, Jean Yeh, and Yeong-Nan Yeh.
\newblock Value-peaks of permutations.
\newblock {\em Electron. J. Combin.}, 17(1):Research Paper 46, 20, 2010.

\bibitem{bradt2024unit}
S.~Alex Bradt, Jennifer Elder, Pamela~E. Harris, Gordon~Rojas Kirby, Eva Reutercrona, Yuxuan Wang, and Juliet Whidden.
\newblock Unit interval parking functions and the r-{F}ubini numbers.
\newblock {\em Matematica}, 3(1):370--384, 2024.

\bibitem{bib:Carlson+}
Joshua Carlson, Alex Christensen, Pamela~E. Harris, Zakiya Jones, and Andres~Ramos Rodriguez.
\newblock Parking functions: Choose your own adventure.
\newblock {\em College Math. J.}, 52(4):254--264, 2021.

\bibitem{cayley1856}
Arthur Cayley.
\newblock {\em On the Theory of the Analytical Forms called Trees}, page 242–246.
\newblock Cambridge Library Collection - Mathematics. Cambridge University Press, 2009.

\bibitem{cerbai2024caylerianpolynomials}
Giulio Cerbai and Anders Claesson.
\newblock Caylerian polynomials.
\newblock {\em Discrete Math.}, 347(12):114177, 2024.

\bibitem{unit_perm}
Lucas Chaves~Meyles, Pamela~E. Harris, Richter Jordaan, Gordon Rojas~Kirby, Sam Sehayek, and Ethan Spingarn.
\newblock Unit-interval parking functions and the permutohedron.
\newblock \href{https://arxiv.org/abs/2305.15554}{arXiv:2305.15554}, to appear in \textit{J. Comb.}, 2023.

\bibitem{desarmenien1992signed}
Jacques D{\'e}sarm{\'e}nien and Dominique Foata.
\newblock The signed eulerian numbers.
\newblock {\em Discrete Math.}, 99(1-3):49--58, 1992.

\bibitem{DiestelGraph}
Reinhard Diestel.
\newblock {\em Graph Theory}.
\newblock Springer Publishing Company, Incorporated, 5th edition, 2017.

\bibitem{Durrett_2019}
Rick Durrett.
\newblock {\em Probability: Theory and Examples}.
\newblock Cambridge Series in Statistical and Probabilistic Mathematics. Cambridge University Press, 5 edition, 2019.

\bibitem{elder2024parking}
Jennifer Elder, Pamela~E. Harris, Jan Kretschmann, and J.~Carlos Mart\'inez~Mori.
\newblock Parking functions, {F}ubini rankings, and {B}oolean intervals in the weak order of {$\mathfrak{S}_n$}.
\newblock {\em J. Comb.}, 16(1):65--89, 2025.

\bibitem{elizalde2024countingpatternavoidingpermutationsbig}
Sergi Elizalde, Johnny~Rivera Jr., and Yan Zhuang.
\newblock Counting pattern-avoiding permutations by big descents.
\newblock \href{https://arxiv.org/abs/2408.15111}{arXiv:2408.15111}, 2024.

\bibitem{FellerProbability}
William Feller.
\newblock {\em An introduction to probability theory and its applications.}
\newblock Wiley mathematical statistics series. Wiley, New York, 1950 - 1966.

\bibitem{foata1974mappings}
Dominique Foata and John Riordan.
\newblock Mappings of acyclic and parking functions.
\newblock {\em Aequationes Math.}, 10:10--22, 1974.

\bibitem{FoaSch70}
Dominique Foata and Marcel-P. Sch\"utzenberger.
\newblock {\em Th\'eorie g\'eom\'etrique des polyn\^omes eul\'eriens}, volume Vol. 138 of {\em Lecture Notes in Math.}
\newblock Springer-Verlag, Berlin-New York, 1970.

\bibitem{FOATA199679}
Dominique Foata and Doron Zeilberger.
\newblock Graphical major indices.
\newblock {\em J. Comput. Appl. Math.}, 68(1):79--101, 1996.

\bibitem{francon_acyclic_1975}
Jean Fran\c{c}on.
\newblock Acyclic and parking functions.
\newblock {\em J. Combin. Theory Ser. A}, 18(1):27--35, January 1975.

\bibitem{peak4}
Jean Fran\c{c}on and G\'erard Viennot.
\newblock Permutations selon leurs pics, creux, doubles mont\'ees et double descentes, nombres d'{E}uler et nombres de {G}enocchi.
\newblock {\em Discrete Math.}, 28(1):21--35, 1979.

\bibitem{fulman2021jointdistributiondescentssigns}
Jason Fulman, Sangchul Lee, Gene~B. Kim, and T.~Kyle Petersen.
\newblock On the joint distribution of descents and signs of permutations.
\newblock {\em Electron. J. Combin.}, 28(3), 2021.

\bibitem{GESSEL1993189}
Ira~M Gessel and Christophe Reutenauer.
\newblock Counting permutations with given cycle structure and descent set.
\newblock {\em J. Combin. Theory Ser. A}, 64(2):189--215, 1993.

\bibitem{GKP}
Ronald~L. Graham, Donald~E. Knuth, and Oren Patashnik.
\newblock {\em Concrete mathematics}.
\newblock Addison-Wesley Publishing Company, Reading, MA, second edition, 1994.
\newblock A foundation for computer science.

\bibitem{ProbabilityGut}
Allan Gut.
\newblock {\em Probability : a graduate course}.
\newblock Springer texts in statistics. Springer Science \& Business Media, New York, 2nd ed. edition, 2013.

\bibitem{bib:HadawayUndergradThesis}
Kimberly~P. Hadaway.
\newblock On {C}ombinatorial {P}roblems of {G}eneralized {P}arking {F}unctions, 2021.
\newblock Williams College Honors Thesis.

\bibitem{DeltaConjecture}
J.~Haglund, J.~B. Remmel, and A.~T. Wilson.
\newblock The delta conjecture.
\newblock {\em Trans. Amer. Math. Soc.}, 370(6):4029--4057, 2018.

\bibitem{ospcoinv}
James Haglund, Brendon Rhoades, and Mark Shimozono.
\newblock Ordered set partitions, generalized coinvariant algebras, and the delta conjecture.
\newblock {\em Adv. Math.}, 329:851--915, 2018.

\bibitem{haiman_conjectures_1994}
Mark~D. Haiman.
\newblock Conjectures on the quotient ring by diagonal invariants.
\newblock {\em J. Algebraic Combin.}, 3(1):17--76, 1994.

\bibitem{OEIS}
OEIS~Foundation Inc.
\newblock The {O}n-{L}ine {E}ncyclopedia of {I}nteger {S}equences, 2025.
\newblock Published electronically at \href{https://oeis.org}{https://oeis.org}.

\bibitem{Konheim1966}
A.~Konheim and B.~Weiss.
\newblock An occupancy discipline and applications.
\newblock {\em SIAM J. Appl. Math.}, 14(6):1266--1274, 1966.

\bibitem{kreweras1980famille}
G.~Kreweras.
\newblock Une famille de polyn\^omes ayant plusieurs propri\'et\'es \'enumeratives.
\newblock {\em Period. Math. Hungar.}, 11(4):309--320, 1980.

\bibitem{LLLSSY23}
Jesse~Campion Loth, Michael Levet, Kevin Liu, Eric~Nathan Stucky, Sheila Sundaram, and Mei Yin.
\newblock Permutation statistics in conjugacy classes of the symmetric group.
\newblock {\href{https://arxiv.org/abs/2301.00898}{arXiv:2301.00898}}, 2023.

\bibitem{peaks5}
Shi-Mei Ma.
\newblock Derivative polynomials and enumeration of permutations by number of interior and left peaks.
\newblock {\em Discrete Math.}, 312(2):405--412, 2012.

\bibitem{MacMahon}
P.~A. MacMahon.
\newblock Two {A}pplications of {G}eneral {T}heorems in {C}ombinatory {A}nalysis.
\newblock {\em Proc. London Math. Soc. (2)}, 15:314--321, 1916.

\bibitem{Marberg_2021}
Eric Marberg.
\newblock Linear compactness and combinatorial bialgebras.
\newblock {\em Electron. J. Combin.}, 28(3):Paper No. 3.9, 47, 2021.

\bibitem{caylerian}
M~Mor and A.S Fraenkel.
\newblock Cayley permutations.
\newblock {\em Discrete Math.}, 48(1):101--112, 1984.

\bibitem{NadTew23}
Philippe Nadeau and Vasu Tewari.
\newblock Remixed {E}ulerian numbers.
\newblock {\em Forum Math. Sigma}, 11:Paper No. e65, 26, 2023.

\bibitem{peaks6}
Kathryn~L. Nyman.
\newblock The peak algebra of the symmetric group.
\newblock {\em J. Algebraic Combin.}, 17(3):309--322, 2003.

\bibitem{peaks7}
T.~Kyle Petersen.
\newblock Enriched {$P$}-partitions and peak algebras.
\newblock {\em Adv. Math.}, 209(2):561--610, 2007.

\bibitem{RemmelWilson}
Jeffrey~B. Remmel and Andrew~Timothy Wilson.
\newblock An extension of {M}ac{M}ahon's equidistribution theorem to ordered set partitions.
\newblock {\em J. Combin. Theory Ser. A}, 134:242--277, 2015.

\bibitem{Rhoades}
Brendon Rhoades.
\newblock Ordered set partition statistics and the delta conjecture.
\newblock {\em J. Combin. Theory Ser. A}, 154:172--217, 2018.

\bibitem{RIORDAN1969408}
John Riordan.
\newblock Ballots and trees.
\newblock {\em J. Combin. Theory Ser. A}, 6(4):408--411, 1969.

\bibitem{sagan2013symmetric}
Bruce~E Sagan.
\newblock {\em The symmetric group: representations, combinatorial algorithms, and symmetric functions}, volume 203.
\newblock Springer Science \& Business Media, 2013.

\bibitem{SMC}
{SageMath Inc.}
\newblock {\em CoCalc Collaborative Computation Online}, 2022.
\newblock \url{ https://cocalc.com/}.

\bibitem{peaks8}
Manfred Schocker.
\newblock The peak algebra of the symmetric group revisited.
\newblock {\em Adv. Math.}, 192(2):259--309, 2005.

\bibitem{bib:DescentsInPF}
Paul R.~F. Schumacher.
\newblock Descents in parking functions.
\newblock {\em J. Integer Seq.}, 21(2):Art. 18.2.3, 8, 2018.

\bibitem{shareshian_eulerian_2010}
John Shareshian and Michelle~L. Wachs.
\newblock Eulerian quasisymmetric functions.
\newblock {\em Adv. Math.}, 225(6):2921--2966, 2010.

\bibitem{shareshian_chromatic_2016}
John Shareshian and Michelle~L. Wachs.
\newblock Chromatic quasisymmetric functions.
\newblock \href{http://arxiv.org/abs/1405.4629}{arXiv:1405.4629}, March 2016.

\bibitem{stanley1976binomial}
Richard~P. Stanley.
\newblock Binomial posets, {M}\"obius inversion, and permutation enumeration.
\newblock {\em J. Combin. Theory Ser. A}, 20(3):336--356, 1976.

\bibitem{stanley1997parking}
Richard~P. Stanley.
\newblock Parking functions and noncrossing partitions.
\newblock {\em Electron. J. Combin.}, 4(2):Research Paper 20, approx. 14, 1997.
\newblock The Wilf Festschrift (Philadelphia, PA, 1996).

\bibitem{stanley1999enumerative}
Richard~P. Stanley.
\newblock {\em Enumerative {C}ombinatorics. {V}olume 2}, volume~62 of {\em Cambridge Stud. Adv. Math.}
\newblock Cambridge University Press, Cambridge, 1999.
\newblock With a foreword by Gian-Carlo Rota and appendix 1 by Sergey Fomin.

\bibitem{stanley2012enumerative}
Richard~P. Stanley.
\newblock {\em Enumerative {C}ombinatorics. {V}olume 1}, volume~49 of {\em Cambridge Stud. Adv. Math.}
\newblock Cambridge University Press, Cambridge, second edition, 2012.

\bibitem{stanley2015catalan}
Richard~P. Stanley.
\newblock {\em Catalan numbers}.
\newblock Cambridge University Press, New York, 2015.

\bibitem{sage}
William\thinspace{}A. Stein et~al.
\newblock {\em {S}age {M}athematics {S}oftware ({V}ersion 9.4)}.
\newblock The Sage Development Team, 2022.
\newblock \url{http://www.sagemath.org}.

\bibitem{OSPstats}
Einar Steingr\'imsson.
\newblock Statistics on ordered partitions of sets.
\newblock {\em J. Comb.}, 11(3):557--574, 2020.

\bibitem{peaks9}
John~R. Stembridge.
\newblock Enriched {$P$}-partitions.
\newblock {\em Trans. Amer. Math. Soc.}, 349(2):763--788, 1997.

\bibitem{Stern_1838}
M.~Stern.
\newblock Aufgaben.
\newblock {\em J. Reine Angew. Math.}, 1838(18):100--100, 1838.

\bibitem{peaks10}
Volker Strehl.
\newblock Enumeration of alternating permutations according to peak sets.
\newblock {\em J. Combinatorial Theory Ser. A}, 24(2):238--240, 1978.

\bibitem{Terquem1838}
O.~Terquem.
\newblock Solution d'un {Probl\`eme} de combinaison.
\newblock {\em J. Math. Pures Appl. (9)}, 1e s{\'e}rie, 3:559--560, 1838.

\bibitem{WACHS199259}
Michelle~L. Wachs.
\newblock An involution for signed eulerian numbers.
\newblock {\em Discrete Math.}, 99(1):59--62, 1992.

\bibitem{YanPFs}
Catherine~H. Yan.
\newblock Parking functions.
\newblock In {\em Handbook of {E}numerative {C}ombinatorics}, Discrete Math. Appl. (Boca Raton), pages 835--893. CRC Press, Boca Raton, FL, 2015.

\end{thebibliography}
\end{document}